\documentclass[11pt, oneside]{amsart}
\usepackage{amssymb,amsmath,amscd,url,geometry,pdflscape, wasysym}
\usepackage{graphicx, multirow, xfrac}
\usepackage{hyperref}

\usepackage[table]{xcolor}
\definecolor{lightgray}{gray}{0.9}

\DeclareRobustCommand{\SkipTocEntry}[5]{}
\def\skiptoc{\addtocontents{toc}{\SkipTocEntry}}

\def\lbar{\overline{l}}
\def\t{\mathfrak t}
\def\fB{\mathfrak{B}}
\def\fs{\mathfrak{s}}
\def\fA{\mathfrak{A}}

\def\group#1{\langle #1\rangle}
\let\gp=\group
\def\zetabar{{\overline{\zeta}}}
\def\etabar{{\overline{\eta}}}
\def\e#1#2{e^{\frac{#1}{#2}\pi i}}
\def\R{\mathbb{R}}
\def\Z{\mathbb{Z}}
\def\Q{\mathbb{Q}}

\def\MaxTorus{\mathcal{T}}
\def\cT{\mathcal{T}}
\def\TT{\mathbb{T}}
\def\Ti{\mathbb{T}^1}
\def\mbar{\overline{m}}
\def\mdot{\dot{m}}
\def\mddot{\ddot{m}}
\def\x{\mathfrak{x}}
\let\mid=:

\newcommand{\C}{\mathbb{C}}
\def\cZ{{\mathcal Z}}
\def\q {{/\!/}}
\def\cG{{\mathcal G}}
\def\cM{{\mathcal M}}

\newtheorem*{theorem2}{Theorem \ref{thm2}}
\newtheorem*{theorem3}{Theorem \ref{thm3}}
\newtheorem*{theorem4}{Theorem \ref{thm4}}
\newtheorem*{theorem5}{Theorem \ref{thm5}}

\newtheorem*{theorem8}{Theorem \ref{thm8}}
\newtheorem{theorem}{Theorem}
\newtheorem{lemma}{Lemma}[section]

\newtheorem*{Conjecture}{Conjecture}
\newtheorem{proposition}[lemma]{Proposition}
\theoremstyle{remark}
\newtheorem{example}[lemma]{Example}
\theoremstyle{definition}
\newtheorem{remark}[lemma]{Remark}
\newtheorem{definition}[lemma]{Definition}

\def\D{\mathcal{D}}
\def\wre#1{\mathop{\wr}_{#1}}
\def\F{\mathbb{F}}

\def\rep{\hbox{$\mathrm{rep}$}}

\title{Centralisers, complex reflection groups and actions in the Weyl group $E_6$}

\author{Graham A. Niblo, Roger Plymen and Nick Wright}
\address{Mathematical Sciences, University of Southampton, SO17 1BJ,  England}
\email{g.a.niblo@soton.ac.uk, r.j.plymen@soton.ac.uk, wright@soton.ac.uk}

\begin{document}

\begin{abstract}
The compact, connected Lie group $E_6$ admits two forms: simply connected and adjoint type.  As we previously established, the Baum-Connes isomorphism relates the two Langlands dual forms, giving a duality between the equivariant K-theory of the Weyl group acting on the corresponding maximal tori.  Our study of the $A_n$ case showed that this duality persists at the level of homotopy, not just homology.
In this paper we compute the extended quotients of maximal tori for the two forms of $E_6$, showing that the homotopy equivalences of sectors established in the $A_n$ case also exist here, leading to a conjecture that the homotopy equivalences always exist for Langlands dual pairs.
In computing these sectors we show that centralisers in the $E_6$ Weyl group decompose as direct products of reflection groups, generalising Springer's results for regular elements, and we develop a pairing between the component groups of fixed sets generalising Reeder's results.
As a further application  we compute the $K$-theory of the reduced Iwahori-spherical $C^*$-algebra of 
the p-adic group $E_6$, which may be of adjoint type or simply connected. 
\end{abstract}

\maketitle

\skiptoc
\section*{Introduction}

The structure of element centralisers in  Weyl groups plays a key role in our understanding of the representation theory of reductive groups of $p$-adic type. For the narrow class of regular elements Springer showed that the centralisers have the structure of a complex reflection group inherited directly from the action on the corresponding regular eigenspace \cite{springer1974regular}. Building on the work of Springer, Brewer \cite{brewer2000complex} classified the irreducible, rank $n$ complex reflection groups in a rank $2n$ real reflection group, exhibiting these as subgroups of centralisers in the automorphism group of the root graph. In the case of elliptic elements Reeder studied the centralisers by constructing a symplectic form on the coinvariant representation \cite{reeder2011elliptic}.

As remarked by Reeder, there is no general theory for the structure of centralisers in Weyl groups, though the classical cases $A_n,B_n,C_n,D_n$ are well understood. In this paper we determine the structure for all the centralisers in the Weyl group of type $E_6$. We provide a  description in terms of generators and relations and give a classification of each centraliser as a product of reflection groups.  In the case of the seven  conjugacy classes of regular elements this recovers Springer's identification of the centralisers as complex reflection groups, though our approach is somewhat different to Springer.

\def\thmvii{Let $W$ be the Weyl group of type $E_6$. For each $w\in W$ the centraliser $Z_W(w)$ decomposes as a product of reflection groups.
}

\begin{theorem}
\label{thm7}\thmvii
\end{theorem}

\noindent For all the elements of order $1$ or $2$ in the Weyl group, the centralisers are real reflection groups.  In all but one of the remaining cases they are truly complex reflection groups. The one exception is the centraliser of the element of type $A_2$.  This is the product of the cyclic group $C_3$ with the group $(S_3\wr C_2)$: the latter of which is not a complex reflection group.  This factor however is a reflection group over the finite field $\F_3$. The structure as reflection groups is outlined in Figure \ref{fig:centraliserinclusions}.

\bigskip

In the case of an elliptic element $w$ in a Weyl group $W$, Reeder studied the centraliser via its actions on the (finite) group of coinvariants $\Gamma/(I-w)\Gamma$ where $\Gamma$ is the root lattice, see \cite{reeder2011elliptic}. Reeder defined a pairing on $\Gamma/(I-w)\Gamma$, for an elliptic element $w$, and showed that for $\Gamma$ a self-dual lattice the pairing is non-degenerate.
Dual to the action on coinvariants one may consider the action of $W$ on the invariant part of the Pontryagin dual of $\Gamma$, which is the maximal torus $\MaxTorus^\vee$ in the Langlands dual of the simply connected form of the Lie group. For non-elliptic elements, the coinvariants and invariants become infinite, leading us to study the finite component groups of the fixed sets in maximal tori $\MaxTorus$ and $\MaxTorus^\vee$ of Langlands dual groups. The centraliser of an element $w\in W$ can be represented by automorphisms of these two finite groups, and we show that these groups and actions are dual.

\bigskip
We will establish this duality in a more general context: let $W$ be a finite group acting orthogonally on a Euclidean vector space $\t$, preserving a lattice $\Gamma$ in $\t$. Let $\Gamma^\vee$ denote the dual lattice in $\t^*$ and let $\MaxTorus=\t/\Gamma,\MaxTorus^\vee=\t^*/\Gamma^\vee$. The fixed sets    $\MaxTorus^w$ and $(\MaxTorus^\vee)^w$ each consist of a product of a sub-torus with a finite subgroup: $\MaxTorus^w\cong \MaxTorus^w_1\times F_w$ and $(\MaxTorus^\vee)^w\cong (\MaxTorus^\vee)^w_1\times F^\vee_w$ where $\MaxTorus^w_1$ and $(\MaxTorus^\vee)^w_1$ denote the identity components of the $w$-fixed sets in $\MaxTorus$ and $\MaxTorus^\vee$ respectively.
(In the case where $W$ is a Weyl group, $\Gamma$ the root lattice and $w\in W$ an elliptic element, the group $F_w^\vee$ is the fixed set for the action of $w$ on $\MaxTorus^\vee$ and the Pontryagin dual of ${F_w^\vee}$ is identified with Reeder's group of coinvariants.)
We will construct a non-degenerate pairing between $F_w, F_w^\vee$ for all elements of the group $W$, allowing us to prove:

\def\thmiv{

Let $W$ be a finite group acting orthogonally on a Euclidean vector space $\t$, preserving a lattice $\Gamma$ in $\t$. Then the dual component group $F^\vee_w$ is (canonically) isomorphic to the Pontryagin dual of $F_w$. Therefore $F^\vee_w$ is (non-canonically) isomorphic to $F_w$.
}

\begin{theorem}\label{Component group duality}
\label{thm4}\thmiv
\end{theorem}

\def\thmviii{
The actions of $Z_W(w)$ on $F_w$ and $F_w^\vee$ are dual.
}
\begin{theorem}
\label{equivariant pairing}\label{thm8}\thmviii
\end{theorem}

Moreover, by exploiting the relationship between invariants and coinvariants we establish the following formula for the cardinality of the component group.

\def\thmv{Let $w\in W$ with $w\neq I$ and let $r$ be the rank of $I-w$.
Let $g$ be the greatest common divisor of the $r\times r$-minors of the matrix of $I-w$ expressed in coordinates with respect to a basis for the dual lattice $\Gamma^\vee$.  Then
\[
g=|F_w|.
\]
}

\begin{theorem}
\label{Cardinality of component groups}
\label{thm5}\thmv
\end{theorem}

The above results allow us to determine the fixed sets and to show that $\MaxTorus^w$ and $(\MaxTorus^\vee)^w$ agree:

\def\thmii{For $w\in W$  the fixed sets $\MaxTorus^w$ and $(\MaxTorus^\vee)^w$ are non-canonically isomorphic as topological groups. In the special case of Weyl groups acting on maximal tori, this gives an isomorphism between the  $w$-fixed sets for the Langlands dual forms.}

\begin{theorem}
\label{homeomorphic fixed sets}
\label{thm2}\thmii
\end{theorem}

\bigskip

For a Weyl group $W$ and maximal torus $\MaxTorus =\t/\Gamma$, the extended affine Weyl group is defined to be $\Gamma \rtimes W$. The $W$-equivariant $K$-theory of the maximal torus $\MaxTorus^\vee$ in the Langlands dual is identified with the $K$-theory of $C^*_r(\Gamma \rtimes W)$, while the $W$-equivariant $K$-homology of $\MaxTorus$ is identified with $K_*^{\Gamma \rtimes W}(\t)$ (see \cite{niblo2018poincare}). Hence the Baum-Connes isomorphism for $\Gamma \rtimes W$ corresponds to a pairing between the equivariant $K$-theory groups $K_W^*(\MaxTorus)$ and $K_W^*(\MaxTorus^\vee)$ which we showed in \cite{niblo2018poincare} is given by a Poincar\'e duality. In particular the pairing makes these two groups isomorphic up to torsion.

\bigskip

Returning to the case of $E_6$, our description of the centralisers and the fixed sets allows us to compute the action of the Weyl group $W$ on the inertia space (cf. Theorem \ref{K-theory calculation}) associated to the action of $W$ on the maximal torus, thereby obtaining a description of the extended quotient.
Using the equivariant Chern character, we compute (up to torsion) the $W$-equivariant $K$-theory for the maximal tori of both the simply connected and adjoint-type Lie groups of type $E_6$, exhibiting the isomorphism between these $K$-theory groups.

\bigskip

The isomorphism between these $K$-theory groups corresponds to a cohomological relation between the corresponding extended quotients for the actions of the Weyl group $W$ on the maximal tori.  
We refine this in the case of $E_6$ by showing that the isomorphism in cohomology arises from a homotopy equivalence at the level of sectors.

\def\thmiii{
Let $W$ be the Weyl group of type $E_6$ and let $\MaxTorus$ and $\MaxTorus^\vee$ denote the corresponding (real) maximal tori. For each $w\in W$ there is a homotopy equivalence
\[
\MaxTorus^w/Z_W(w)\sim (\MaxTorus^\vee)^w/Z_W(w).
\]
}

\begin{theorem}
\label{thm3}\thmiii
\end{theorem}

This stratification phenomenon was introduced and studied in the case of $A_n$ in \cite{niblo2019stratified},
and we conjecture that this demonstrates a general principle for any compact connected semisimple Lie group:

\begin{Conjecture}[Comparison of sectors]
  For any compact connected semisimple Lie group, the homeomorphisms between $\MaxTorus^w, (\MaxTorus^\vee)^w$ provided by Theorem \ref{homeomorphic fixed sets} descend to  homotopy equivalences
\[
\MaxTorus^w/Z_W(w) \sim (\MaxTorus^\vee)^w/ Z_W(w).
\]
\end{Conjecture}

\medskip

In Section \ref{applications}, we provide two main applications.   The first is the computation of the $K$-theory for the group $C^*$-algebras of the two extended affine Weyl groups of type $E_6$ 
showing that these agree.    
The second is a geometric description of the set of tempered representations in the Iwahori-spherical block of the $p$-adic adjoint group $E_6$.  For each of these applications, 
full use is made of our results in Table 2.

\medskip
The paper is organised as follows:
\setcounter{tocdepth}{1}
\renewcommand\contentsname{}
\vspace{-1cm}
\tableofcontents

\skiptoc\section*{Acknowledgement}The authors would like to thank the referee for their comments and the community of developers and mathematicians who write and maintain the GAP computer algebra system. While none of the arguments given here rely on computer calculations, and the conjugacy class representatives were built carefully by hand to simplify the construction of centralisers and their actions, we found the Lie Algebra data package provided by Willem Adriaan de Graaf and Thomas Breuer very helpful in our initial investigations.

\section{Component groups of fixed sets}\label{Component groups}

In this section we will consider the general case of a finite group $W$ acting orthogonally on a Euclidean vector space $\t$ of dimension $n$, preserving a lattice $\Gamma$. The action thus descends to the torus $\MaxTorus=\t/\Gamma$. Let $w$ be an element of $W$ and $\MaxTorus^w$ the fixed set of the action of $w$ on $\MaxTorus$.  Since the action is by automorphisms of the group $\MaxTorus$, in each case the fixed set is a (closed) subgroup of $\MaxTorus$ and is therefore isomorphic to the direct product of a finite abelian group with the identity component of $\MaxTorus^w$, which we denote $\MaxTorus^w_1$.  We note that $\MaxTorus^w_1$ is the image under the exponential map of the $1$-eigenspace of $w$ acting on the Lie algebra $\mathfrak t$. Let $F_w$ denote the group of components of the $w$-fixed set, that is
\[
F_w=\MaxTorus^w/ \MaxTorus^w_1
\]
so that $\MaxTorus^w\cong \MaxTorus^w_1\times F_w$.

We use the inner product to identify $\t$ with $\t^*$. Note that this identifies the $W$-action on $\t$ with the dual action on $\t^*$, preserving the dual lattice
\[
\Gamma^\vee=\{x\in \t : \langle x,y\rangle \in \Z,\; \forall y\in\Gamma\}.
\]
This induces an action of $W$ on the dual torus $\MaxTorus^\vee=\t/\Gamma^\vee$.
Let $F^\vee_w$ denote the component group for the action on the dual torus: $F^\vee_w=(\MaxTorus^\vee)^w/ (\MaxTorus^\vee)^w_1$.

\begin{theorem4}
\thmiv
\end{theorem4}

To prove the theorem we will construct a $\Ti$-valued pairing between the Pontryagin duals $\widehat{F_w}$ and $\widehat{F_w^\vee}$ thus establishing an isomorphism
\[
\widehat{F_w}\cong\widehat{\widehat{F_w^\vee}}\cong F_w^\vee.
\]

The Pontryagin dual $\widehat{F_w}$ is canonically isomorphic to the set of characters on $\MaxTorus^w$ which are trivial on the subgroup $\MaxTorus^w_1$.
The group of characters on $\MaxTorus$ is identified with the lattice $\Gamma^\vee$ via
\[
\chi_x(y+\Gamma):=e^{2\pi i \langle x,y\rangle}.
\]
The characters on $\MaxTorus^w$ are given by the quotient:
\[
\widehat{\MaxTorus^w}=\Gamma^\vee/((I-w)\Gamma^\vee)
\]
Hence the set of characters on $\MaxTorus^w$ vanishing on $\MaxTorus^w_1$ is given by
\[
\{x \in\Gamma^\vee : \chi_x \equiv 1 \text{ on } \MaxTorus^w_1\}/((I-w)\Gamma^\vee).
\]
Note that $\chi_x \equiv 1 \text{ on } \MaxTorus^w_1$ if and only if $x$ is orthogonal to the fixed set $\t^w$, so we obtain the identification:
\[
\widehat{F_w}\cong \{x \in\Gamma^\vee : x\perp \t^w\}/((I-w)\Gamma^\vee).
\]
Likewise for the dual ${F_w^\vee}$ we have
\[
\widehat{F^\vee_w}\cong \{y \in\Gamma : y\perp \t^w\}/((I-w)\Gamma).
\]
We introduce the notation
\begin{align*}
    \x_w&=\{ y\in \t : y\perp \t^w\}\\
    \x_w(\Gamma)&=\x_w\cap \Gamma\\
    \x_w(\Gamma^\vee)&=\x_w\cap \Gamma^\vee.
\end{align*}

Following Reeder, let $m$ denote the minimal polynomial of $w$ and let $\mdot(t)=\dfrac{m(t)-m(1)}{t-1}$. Reeder restricts to the elliptic case, i.e.\ where $m(1)\neq 0$. In the non-elliptic case we will define $\mddot(t)=\dfrac{\mdot(t)-\mdot(1)}{t-1}$ and in general we set
\[
\mbar=\begin{cases}\mdot & m(1)\neq 0\\\mddot & m(1)= 0.\end{cases}
\]

\begin{definition}
For $x\in \Gamma^\vee$ and $y\in \Gamma$ we define a twisted pairing
\[
\langle x,y\rangle_w:=\langle x,\mbar(w)y\rangle.
\]
\end{definition}

We will consider this integer valued pairing modulo $\mu$ where
\[
\mu=\begin{cases}m(1) &\text{if } m(1)\neq 0\\\mdot(1) & \text{otherwise}.\end{cases}
\]
Note that $\mu$ is always non-zero since $m$ is the minimal polynomial, and $w$ is a normal operator. 

The following Lemma shows that $\mbar(w)$ provides an inverse in the elliptic case, and a partial inverse in general, to $I-w$.
\begin{lemma}\label{mu projection}
The operator $\mbar(w)(I-w)$ is $\mu p_{\x_w}$ where  $p_{\x_w}$ denotes the orthogonal projection onto $\x_w$.
\end{lemma}
\begin{proof}
Clearly $\mbar(w)(I-w)=0$ on the fixed set $\t^w$ which is the orthogonal complement of $\x_w$, so it suffices to show that $\mbar(w)(I-w)y=\mu y$ for $y\in \x_w$.

In the elliptic case the result is immediate from the fact that $\mdot(t)(1-t)=\mu-m(t)$ and $m(w)=0$.

In the non-elliptic case we have $\mddot(t)(1-t)=\mu-\mdot(t)$.  The operator $\mdot(w)$ satisfies the equation
\[
\mdot(w)(I-w)=m(1)-m(w)=0.
\]
The image of $I-w$ is precisely $\x_w$ since $w$ is a normal operator and $\t^w=\ker(I-w)$ hence for $y\in \x_w$
\[
\mddot(w)(I-w)y=\mu y-\mdot(w)y=\mu y.
\]
\end{proof}

\begin{lemma}
The twisted pairing $\langle-,-\rangle_w$ descends to a well-defined pairing
\[
\widehat{F_w}\times\widehat{F^\vee_w}\to \Z/\mu\Z.
\]
\end{lemma}
\begin{proof}
Recall that $\widehat{F_w}$ is identified with $\x_w(\Gamma^\vee)/(I-w)\Gamma^\vee$ and $\widehat{F^\vee_w}$ is identified with $\x_w(\Gamma)/(I-w)\Gamma$.

Let $x\in \x_w(\Gamma^\vee)$ and $y\in \Gamma$.  Then
\begin{align*}
    \langle x,(I-w)y\rangle_w 
    &= \langle x,\mbar(w)(I-w)y\rangle\\
    &= \langle x,\mu p_{\x_w} y\rangle\\
    &= \langle p_{\x_w} x,\mu y\rangle\\
     &= \mu\langle x,y\rangle\in \mu\Z.
\end{align*}
Similarly if $x\in \Gamma^\vee$ and $y\in \x_w(\Gamma)$ then 
\begin{align*}
    \langle (I-w)x,y\rangle_w 
    &= \langle (I-w)x,\mbar(w)y\rangle\\
    &= \langle (w^{-1}-I)wx,\mbar(w)y\rangle\\
    &= \langle wx,(w-I)\mbar(w) y\rangle\\
    &= \langle wx,-\mu p_{\x_w}y\rangle\\
     &= -\mu\langle wx,y\rangle\in \mu\Z.
\end{align*}
Hence the pairing is well defined on the quotients.
\end{proof}

We note that since $\x_w(\Gamma)$ contains $(I-w)\Gamma$ it spans the space $\x_w=(I-w)\t$.  In particular $\x_w(\Gamma)$ is a lattice in the space $\x_w$, and similarly for $\x_w(\Gamma^\vee)$.

We consider the lattices $(\x_w(\Gamma))^\vee=\{x\in \x_w : \langle x,\x_w(\Gamma)\rangle \subseteq \Z\}$ and $(\x_w(\Gamma^\vee))^\vee=\{y\in \x_w : \langle \x_w(\Gamma^\vee),y\rangle \subseteq \Z\}$.  Since $\x_w(\Gamma^\vee)$ lies in $\Gamma^\vee$ it pairs integrally with $\Gamma$ and hence $\x_w(\Gamma^\vee)\subseteq (\x_w(\Gamma))^\vee$.  Similarly $\x_w(\Gamma)\subseteq (\x_w(\Gamma^\vee))^\vee$.
Let
\begin{align*}
    \t^w(\Gamma)&=\t^w\cap \Gamma,\\ \t^w(\Gamma^\vee)&=\t^w\cap \Gamma^\vee.
\end{align*}

\begin{lemma}\label{Projection of Gamma^vee}
The projection $p_{\x_w}$ gives  maps $\Gamma^\vee\to (\x_w(\Gamma))^\vee$ and $\Gamma\to (\x_w(\Gamma^\vee))^\vee$ which induce isomorphisms
\[
\Gamma^\vee/\t^w(\Gamma^\vee)\cong (\x_w(\Gamma))^\vee.
\]
\[
\Gamma/\t^w(\Gamma)\cong (\x_w(\Gamma^\vee))^\vee.
\]
\end{lemma}
\begin{proof}
If $x\in \Gamma^\vee$ then $\langle x,y\rangle\in \Z$ for all $y\in \x_w(\Gamma)$.  Now
\[
\langle p_{\x_w}x,y\rangle=\langle x,p_{\x_w}y\rangle=\langle x,y\rangle\in \Z
\]
and $p_{\x_w}x\in \x_w$ so $p_{\x_w}x\in (\x_w(\Gamma))^\vee$.  Clearly this map has kernel $\t^w(\Gamma^\vee)$.

It remains to check surjectivity.  Since $\x_w(\Gamma)$ is a lattice in $\x_w$ and is the intersection of $\x_w$ with $\Gamma$ it follows that $\x_w$ is complemented in $\Gamma$. Picking a complement for $\x_w(\Gamma)$, let $\pi:\Gamma\to \x_w(\Gamma)$ denote the retraction obtained by killing the complement.

Given any $x\in (\x_w(\Gamma))^\vee$ consider the homomorphism from $\Gamma$ to $\Z$ defined by
\[
y\mapsto \langle x,\pi(y)\rangle.
\]
Since there is a perfect pairing between $\Gamma$ and $\Gamma^\vee$, there exists $x'\in \Gamma^\vee$ such that $\langle x,\pi(y)\rangle=\langle x',y\rangle$.  Now for all $y\in \x_w(\Gamma)$ we have
\[
\langle p_{\x_w}x',y\rangle = \langle x',p_{\x_w}y\rangle = \langle x',y\rangle = \langle x,\pi(y)\rangle=\langle x,y\rangle.
\]
Thus $x=p_{\x_w}x'$ which is in the image.

Exchanging the roles of $\Gamma$ and $\Gamma^\vee$ gives the dual case.
\end{proof} 

\begin{proposition}
The pairing
\[
\widehat{F_w}\times\widehat{F^\vee_w}\to \Z/\mu\Z.
\]
is left and right non-degenerate.
\end{proposition}

\begin{proof}
Let $x\in\x_w(\Gamma^\vee)$ and suppose that $\langle x,y\rangle_w\cong 0$ mod $\mu$ for all $y\in \x_w(\Gamma)$, that is $\langle x,\frac{1}{\mu}\mbar(w)y\rangle\in \Z$ for all $y$ in $\x_w(\Gamma)$.

Hence $\frac{1}{\mu}\mbar(w)^* x$ pairs integrally with all $y$ in $\x_w(\Gamma)$, and lies in the space $\x_w$, so $\frac{1}{\mu}\mbar(w)^* x\in (\x_w(\Gamma))^\vee$.

By Lemma \ref{Projection of Gamma^vee} $(\x_w(\Gamma))^\vee=p_{\x_w}\Gamma^\vee$ and by Lemma \ref{mu projection}
\[
\mu p_{\x_w}=\mbar(w)(I-w) = (\mbar(w)(I-w))^*
\]
so we have
\[
\mbar(w)^* x \in \mu p_{\x_w}\Gamma^\vee = \mbar(w)^*(I-w)^*\Gamma^\vee
\]
since $\mbar(w)$ and $I-w$ commute.

Now $x,(I-w)^*\Gamma^\vee$ lie in $\x_w$ and $\mbar(w)^*$ is injective on this space (again by Lemma \ref{mu projection}) giving
\[
x\in (I-w)^*\Gamma^\vee = (I-w^{-1})\Gamma^\vee=(w-I)w^{-1}\Gamma^\vee=(I-w)\Gamma^\vee.
\]
Thus the pairing is non-degenerate on the left.

Now let $y\in \x_w(\Gamma)$ and suppose that $\langle x,y\rangle_w\cong 0$ mod $\mu$ for all $x\in \x_w(\Gamma^\vee)$, that is $\langle x,\frac{1}{\mu}\mbar(w)y\rangle\in \Z$ for all $x$ in $\x_w(\Gamma^\vee)$.

Hence $\frac{1}{\mu}\mbar(w)y\in (\x_w(\Gamma^\vee))^\vee=p_{\x_w}\Gamma$. Then
\[
\mbar(w)y\in \mu p_{\x_w}\Gamma = \mbar(w)(I-w)\Gamma
\]
and as $\mbar(w)$ is injective on $\x_w$ we have $y\in (I-w)\Gamma$. Hence the pairing is also non-degenerate on the right.
\end{proof}

\begin{proof}[Proof of Theorem \ref{Component group duality}]
As noted above it suffices to prove that $\widehat{F_w}$ is the Pontryagin dual of $\widehat{F_w^\vee}$.

Define a pairing $\widehat{F_w}\times\widehat{F^\vee_w}\to \Ti$ by
\[
\chi(x+(I-w)\Gamma^\vee, y+(I-w)\Gamma):=e^{\frac{2\pi i}{\mu} \langle x,y\rangle_w}.
\]
Fixing an element of $\widehat{F_w}$ this gives a character on $\widehat{F_w^\vee}$, hence we have a homomorphism from $\widehat{F_w}$ to the Pontryagin dual of $\widehat{F_w^\vee}$.  Left non-degeneracy of the pairing $\langle-,-\rangle_w$ implies injectivity of this homomorphism.

Likewise fixing an element of $\widehat{F_w^\vee}$ gives a character on $\widehat{F_w}$, hence we have a homomorphism from $\widehat{F_w^\vee}$ to the Pontryagin dual of $\widehat{F_w}$.  Right non-degeneracy again implies injectivity.

Combining these two maps we deduce that
\[
|\widehat{F_w}|\leq |\widehat{\widehat{F_w^\vee}}|=|\widehat{F_w^\vee}|\leq |\widehat{\widehat{F_w}}|=|\widehat{F_w}|.
\]
Since the cardinalities are equal the injections must be isomorphisms.
\end{proof}

The following result gives a tool for computing the cardinality of these groups, which is relevant in the computation of $K$-theory.
It may be useful to bear in mind the following example.

\begin{example}
The conjugacy class representative $s_0s_1s_5s_3$ of type $A_1^4$ is represented by the matrix 
\[
M=\begin{pmatrix}
-1 & 1 & 0 & 0 & 0 & -1 \\
0 & 1 & 0 & 0 & 0 & -2 \\
0 & 1 & -1 & 1 & 0 & -2 \\
0 & 0 & 0 & 1 & 0 & -2 \\
0 & 0 & 0 & 1 & -1 & -1 \\
0 & 0 & 0 & 0 & 0 & -1 \\
\end{pmatrix}
\]
with respect to the lattice $\Gamma$ and its transpose with respect to the lattice $\Gamma^\vee$. The rank of $I-M$ is $4$ and the greatest common divisors of the $4\times 4$ minors is easily seen to be $4$ for both $I-M$ and $I-M^T$. As we will establish in the next theorem, this computes the order of the component group of the  corresponding fixed set.

This example is simplified by the fact that the operator is self adjoint, ensuring that the dual matrices over $\Gamma$ and $\Gamma^\vee$ are transpose. In general, the matrices with respect to the dual bases are related by the formula $(I-M)\rightarrow (I-M^{-1})^T$, however the argument given in Theorem \ref{thm2} below shows that the $\gcd$ of the corresponding minors is still the same in these two matrices.
\end{example}

\begin{theorem5}
\thmv
\end{theorem5}

\begin{proof}
Let $i=|F_w|=[\x_w(\Gamma^\vee):(I-w)\Gamma^\vee]$.

Let $v_1,\dots, v_r$ be a basis for $(I-w)\Gamma^\vee$, and let $w_1\dots,w_r$ be a basis for $\x_w(\Gamma^\vee)$.

Expressing these vectors in coordinates with respect to a basis for $\Gamma^\vee$, each element of the basis $\{v_i\}$ lies in the integer column span of $I-w$, so the greatest common divisor of the $r\times r$-minors of the matrix $(v_1|\dots| v_r)$ must be divisible by $g$.  But conversely each column of $I-w$ can be written in terms of the basis, hence we deduce that $g$ equals the greatest common divisor of the $r\times r$-minors of the matrix $(v_1|\dots| v_r)$.

The subgroup $\x_w(\Gamma^\vee)=\x_w\cap\Gamma^\vee$ has a complement $\Lambda$ in $\Gamma^\vee$ and the index $i$ of $(I-w)\Gamma^\vee$ in the lattice $\x_w(\Gamma^\vee)$ equals the index of $(I-w)\Gamma^\vee\oplus \Lambda$ in $\Gamma^\vee$. Let $x_1\dots,x_{n-r}$ be a basis of $\Lambda$, so $w_1\dots,w_r,x_1,\dots x_{n-r}$ gives a basis for $\Gamma^\vee$ and hence the corresponding $n\times n$ matrix has determinant 1.

The elements $v_1,\dots,v_r, x_1,\dots, x_{n-r}$ form a basis for $(I-w)\Gamma^\vee\oplus \Lambda$ which has index $i$ in $\Gamma^\vee$, hence the corresponding matrix has determinant $i$.  But the determinant can be built from the minors of $(v_1|\dots| v_r)$, so must be divisible by $g$.  That is $g | i$.

Now consider exterior products:  The determinants are given by $v_1 \wedge \dots  \wedge v_r \wedge x_1 \wedge \dots  \wedge x_{n-r}$ and by $w_1 \wedge \dots  \wedge w_r \wedge x_1 \wedge \dots  \wedge x_{n-r}$ thus
\[
v_1  \wedge \dots  \wedge v_r \wedge x_1 \wedge \dots  \wedge x_{n-r} = i (w_1 \wedge  \dots  \wedge w_r \wedge x_1 \wedge \dots  \wedge x_{n-r})
\]

As $v_1,\dots, v_r$ and $w_1,\dots, w_r$ span the same $r$-dimensional subspace of $\t^*$ it follows that $v_1 \wedge  \dots  \wedge v_r$ is a multiple of $w_1  \wedge \dots  \wedge w_r$, and from the above equation we see that the coefficient is $i$:
\[
v_1 \wedge \dots  \wedge v_r = i (w_1  \wedge \dots  \wedge w_r).
\]

The coefficients of $v_1 \wedge \dots  \wedge v_r$ in the standard basis $\{ e_{i_1} \wedge \dots  \wedge e_{i_r} \}$ for the exterior algebra of $\Gamma^\vee$ are precisely the $r\times r$ minors.  But the above equation tells us that as $w_1 \wedge \dots  \wedge w_r$ has integer coefficients, the coefficients of $v_1 \wedge \dots \wedge v_r$ must be divisible by $i$.  Hence $i | g$.

Therefore $i=g$ as claimed.
\end{proof}

\begin{remark}
We note that since $|F_w|=|F_w^\vee|$ by Theorem \ref{Component group duality} we may use the greatest common divisors of minors of the matrix of $I-w$ with respect to a basis for either the lattice $\Gamma$ or the lattice $\Gamma^\vee$ to compute the cardinality of the groups $F_w,F_w^\vee$.
\end{remark}

As a consequence of Theorem \ref{Component group duality}  we have the following:

\begin{theorem2}
\thmii
\end{theorem2}

\begin{proof}
The fixed sets are isomorphic to $\MaxTorus_1^w\times F_w, (\MaxTorus^\vee_1)^w\times F^\vee_w$ and we know that the component groups are isomorphic by \ref{Component group duality}  so we only need to show that the identity components are isomorphic. This follows from the fact that these are precisely the image of the fixed sets in $\t\cong \t^*$ under the respective quotient maps.
\end{proof}

We now turn to the actions of the centraliser $Z_W(w)$ on the groups $F_w$ and $F_w^\vee$, induced by the actions of $Z_W(w)$ on the $w$-fixed sets $\MaxTorus^w$ and $(T^\vee)^w$.  Although the groups $F_w$ and $F^\vee_w$ are isomorphic, there is in general no $Z_W(w)$-\emph{equivariant} isomorphism between the component groups.  Instead we will show that the actions are dual. Nonetheless even though these dual actions can be very different we will see (cf. Section \ref{section: adjoint fixed sets}) that the numbers of orbits of the centraliser on $F_w$ and $F^\vee_w$ are the same.  This is an instance of our conjecture and can be seen clearly in the dual actions of $G_{25}$ on $\mathbb{F}_3^3$ in case (\ref{Fixed set A_2^3}) in Section \ref{fixed sets}.

\begin{theorem8}
\thmviii
\end{theorem8}

\begin{proof}
For $s\in F_w,t\in F_w^\vee\cong \widehat{F_w}$ denote the corresponding pairing by $\langle s,t\rangle^w$. We must show:
\[
\langle gs,gt\rangle^w = \langle s,t\rangle^w \qquad \forall s\in F_w,t\in F_w^\vee, g\in Z_W(w).
\]

Recall that $\widehat{F_w}\cong \x_w(\Gamma^\vee)/(I-w)\Gamma^\vee$ and $\widehat{F_w^\vee}\cong \x_w(\Gamma)/(I-w)\Gamma$.  By construction these isomorphisms are equivariant for the action of $Z_W(w)$.

The pairing of $\widehat{F_w}$ and $\widehat{F_w^\vee}$ is defined by
\[
\chi(x+(I-w)\Gamma^\vee, y+(I-w)\Gamma):=e^{\frac{2\pi i}{\mu} \langle x,y\rangle_w}.
\]
and for $g\in Z_W(w)$ we have $\chi(gx+(I-w)\Gamma^\vee, gy+(I-w)\Gamma):=e^{2\pi i \langle gx,gy\rangle_w}$.  Now note that
\[
\langle gx,gy\rangle_w = \langle gx,\mbar(w)gy\rangle =\langle gx,g\mbar(w)y\rangle
\]
since $g$ is in the centraliser of $w$, so as $W$ acts isometrically on $\t$ we have $\langle gx,gy\rangle_w=\langle x,y\rangle_w$.

Now for $s\in F_w,t\in F_w^\vee$ the pairing of $gs$ with $gt$ is defined to be the pairing of $gs$ with the image of $gt$ under the composition
\[
F_w^\vee \cong \widehat{\widehat{F_w^\vee}} \cong \widehat{F_w}.
\]
The first isomorphism is tautologically equivariant while the second is equivariant by the above calculation.  Letting $\psi$ denote the image of $t$ under this composition we have
\[
\langle gs,gt\rangle^w = \langle g\psi,gs\rangle = \langle \psi,s\rangle=\langle s,t\rangle^w.
\]
\end{proof}

Bearing in mind the example of $E_6$ we now suppose that $\Gamma^\vee$ refines $\Gamma$ with quotient $\cZ=\Gamma^\vee/\Gamma$ cyclic of prime order.  Moreover we assume that the induced action of $W$ on the quotient $\cZ$ is trivial.  It follows that $\cZ$ lies in each of the fixed groups $\MaxTorus^w$.

The $w$-fixed set in the dual torus $\MaxTorus^\vee=\t/\Gamma^\vee = \MaxTorus/\cZ$ contains the image of $\MaxTorus^w$ under the quotient map $\MaxTorus\to \MaxTorus^\vee$. In particular $(\MaxTorus^\vee)^w_1$ is the image of $\MaxTorus^w_1$ under the quotient map, since the identity component of the fixed set is precisely the image of the fixed set $\t^w$ under the respective exponential map.  Hence the quotient map induces a map from $F_w=\MaxTorus^w/\MaxTorus^w_1$ to $F_w^\vee=(\MaxTorus^\vee)^w/(\MaxTorus^\vee)^w_1$.

We have two cases:  either $\cZ$ lies in $\MaxTorus^w_1$ or $\cZ\cap \MaxTorus^w_1$ is trivial.  In the former case the map from $F_w$ to $F_w^\vee$ is injective.  It is thus an isomorphism since the groups have the same cardinality by Theorem \ref{Component group duality}.  Hence in this case the fixed set $(\MaxTorus^\vee)^w$ is precisely the image of the fixed set $\MaxTorus^w$ under the quotient by $\cZ$, that is:
\[
(\MaxTorus^\vee)^w=(\MaxTorus^\vee)^w_1\times F_w^\vee\cong \MaxTorus^w_1/\cZ\times F_w.
\]

In the latter case the quotient map $\MaxTorus^w\to (\MaxTorus^\vee)^w$ gives an isomorphism of the identity components and takes $F_w$ to a subgroup of $F_w^\vee$, which, again by Theorem \ref{Component group duality}, must have index $|\cZ|$ in $F_w^\vee$.  Hence in this case the image of $\MaxTorus^w$ in $(\MaxTorus^\vee)^w$ is
\[
\MaxTorus^w_1\times F_w/\cZ
\]
which has index $|\cZ|$ in the fixed set $(\MaxTorus^\vee)^w$.

\bigskip

\begin{remark}\label{ramified component remark}
In the case that $\cZ<F_w$ the $Z_W(w)$-orbits in $F_w$ include $|\cZ|$ singletons.  In this case, dually the orbits in $F_w^\vee$ are partitioned into $|\cZ|$ sets.  To see this we note that the inclusion of $\cZ$ in $F_w$ induces a quotient $\widehat{F_w}\to \widehat{\cZ}$.  Since $F_w^\vee \cong \widehat{F_w}$ we have a quotient map $F_w^\vee\to \widehat{\cZ}$, given by the pairing of $F_w^\vee$ with $\cZ$.  By  Theorem \ref{equivariant pairing} the pairing is equivariant but $\cZ$ is fixed by $Z_W(w)$, hence the map $F_w^\vee\to \widehat{\cZ}$ is invariant under the action of $Z_W(w)$. Each $Z_W(w)$-orbit in $F_w^\vee$ thus lies in a coset of the kernel of the map $F_w^\vee\to \widehat{\cZ}$.
\end{remark}

\section{Elements, conjugacy classes and centralisers}\label{elements}

In this section we will provide a list of carefully selected representatives for the conjugacy classes in the Weyl group of type $E_6$ together with key properties of these elements.  We will also introduce a number of special elements of the Weyl group that will play a key role in understanding the centraliser subgroups and their actions.

The elements $s_1, \ldots, s_6$ denote the standard simple reflections generating the Weyl group $W$, with corresponding roots denoted $r_1,\dots, r_6$.  We define $s_0$  to be the reflection corresponding to the root
\[
r_0:=-r_1-2r_2-3r_3-2r_4-r_5-2r_6
\]
which is the negation of the root of highest weight.
This is at $120^{\circ}$ to the simple root $r_6$, so the group generated by $s_0$ and $s_6$ is isomorphic to the symmetric group $S_3$.  On the other hand, $r_0$ is orthogonal to the roots $r_1, \ldots, r_5$, so $s_0$ commutes with the reflections $s_1, \ldots s_5$. 

These relationships can be summarised in the following extension of the Dynkin diagram for $E_6$. We remark that the reflection $s_0$ is the linear part of the additional simple reflection in the affine Weyl group $\tilde{E_6}$, and the extended diagram is isomorphic to the Dynkin diagram of this group.

\begin{figure}[ht] 
   \centering
   \includegraphics[width=2in]{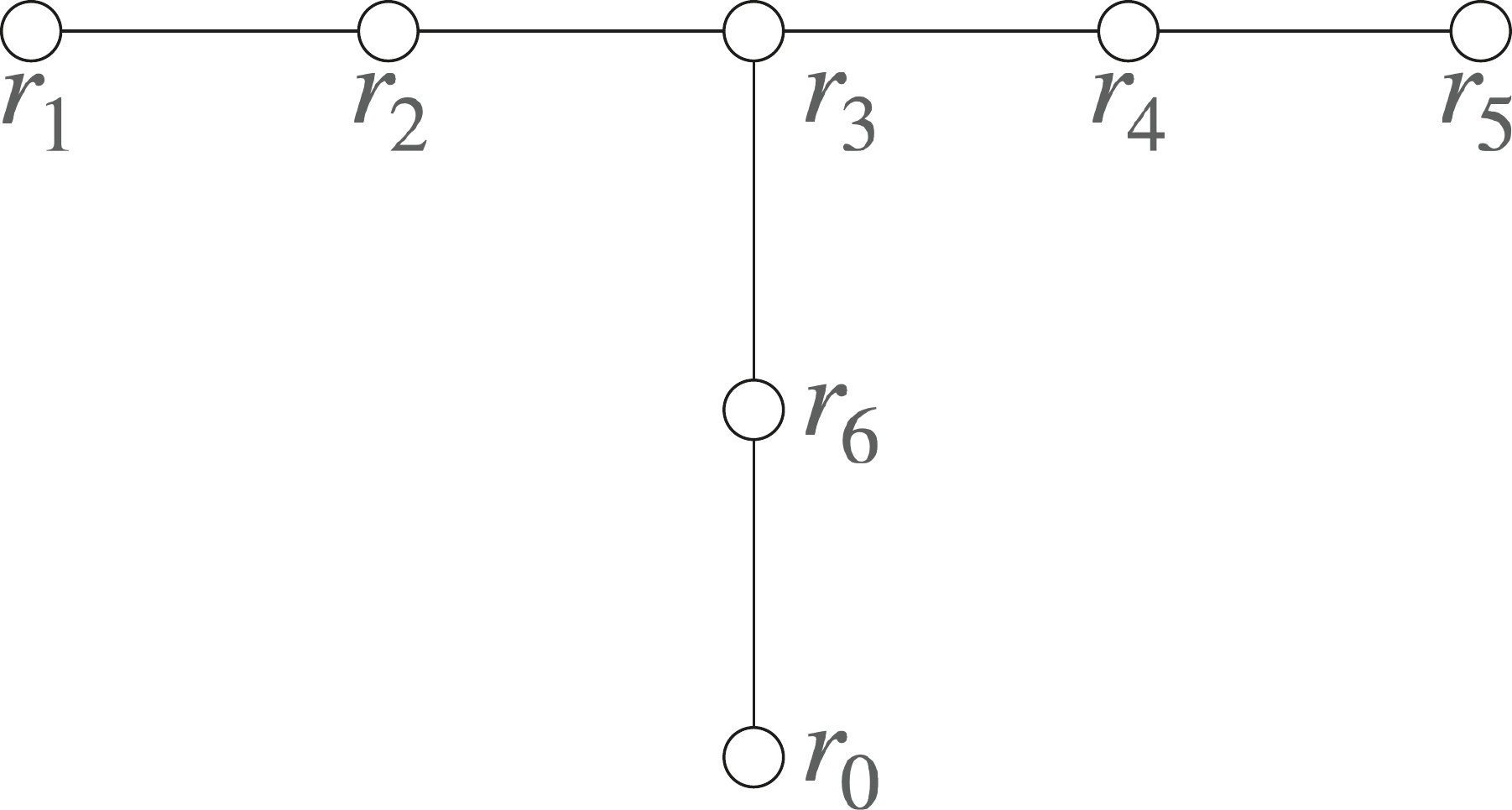} 
   \caption{Coxeter-Dynkin Diagram of the roots $r_0,\dots,r_6$.}
   \label{fig:ExtendedDynkinDiagramE6Tilde}
\end{figure}

This diagram demonstrates additional symmetries of the set of roots for $E_6$, which is the reason that including $r_0$ is so valuable. In choosing representatives for the conjugacy classes we also make use of the reflection $T=s_3^{s_2s_4s_6}$ corresponding to the root $r_T:=-(r_2+r_3+r_4+r_6)=1/2(r_0+r_3+r_1+r_5)$. The name is chosen to reflect the geometry of the roots $r_2,r_3,r_4,r_6$ in the Dynkin diagram, and the symmetry of this set in the extended diagram again motivates the inclusion of this element.

The roots $r_0, r_1, r_5$ and $-r_T$ form a  Dynkin diagram of type $D_4$ with $-r_T$ as the central vertex, hence the group $\group{s_0, s_1, s_5, T}$ is  a Weyl group of type $D_4$. We remark in passing that his group contains $s_3$ as the reflection with root $-2r_T+r_0+r_1+r_5$.

The centralisers will be built from an elementary part along with a small number of additional elements $u_1,u_2,u_3,T,Ts_3$.

\subsection{The special elements}\label{section:special}

We now define elements $u_1, u_2$ in $W$ which generate a copy of $S_3$ acting (by conjugation) as a dihedral group on the reflections $s_0, \ldots s_6$ as illustrated on the following figure. 

\begin{figure}[ht] 
   \centering
   \includegraphics[width=70mm]{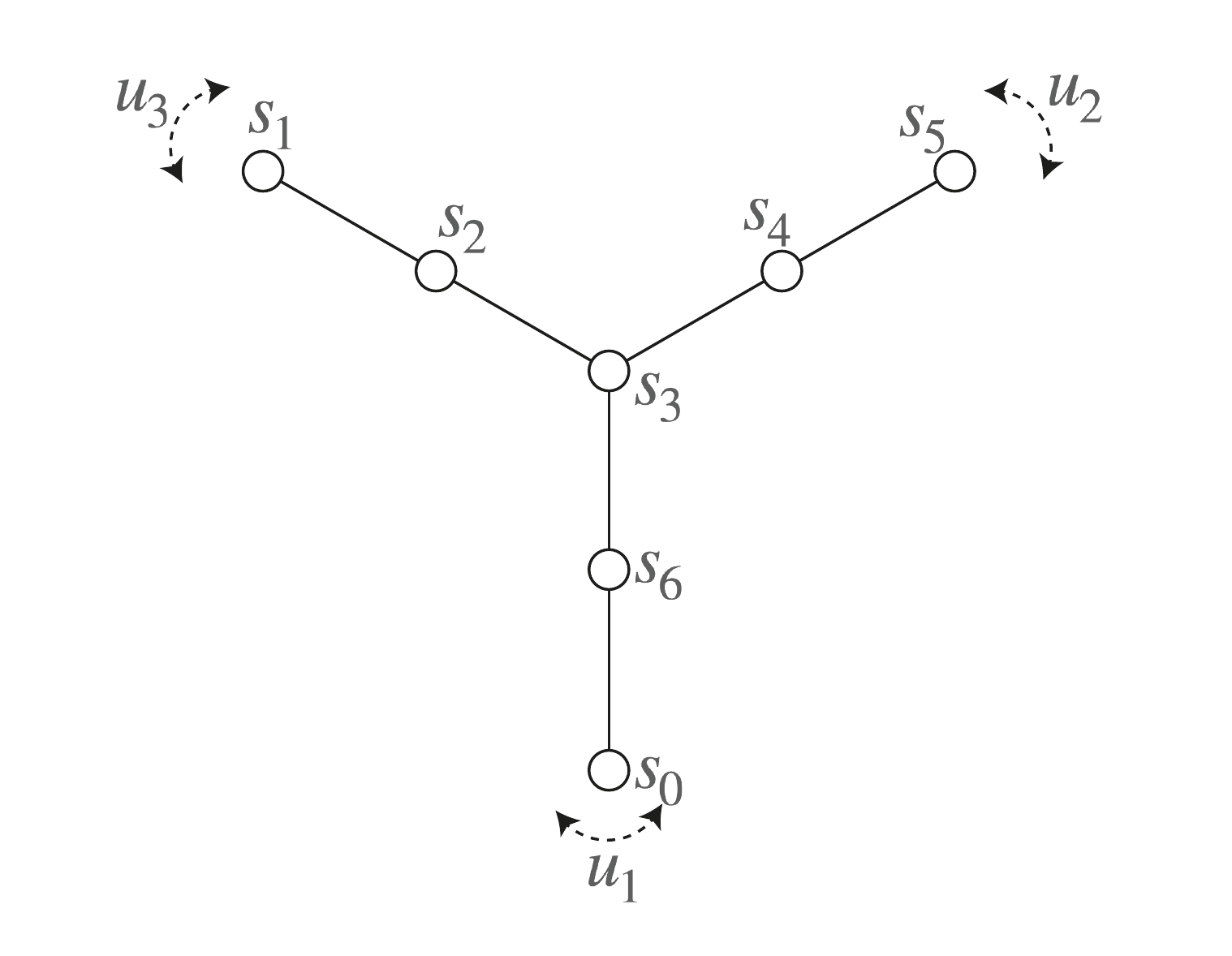} 
   \caption{The action of the special elements $u_i$.}
   \label{fig:specelts}
\end{figure}

The element $u_1$ is defined as the product of four commuting reflections with roots 

\[
\{r_0,\,r_3,\, r_2+r_3+r_4,\, r_1+r_2+r_3+r_4+r_5\}.
\]
It is easy to see that this element negates the roots $r_0, r_3$, so $s_0^{u_1}=s_0$ and $s_3^{u_1}=s_3$.  Acting with $u_1$ on the root $r_6$ gives the sum
\[
r_6+r_0+r_3+(r_2+r_3+r_4)+(r_1+r_2+r_3+r_4+r_5)=-r_6
\]
hence $s_6^{u_1}=s_6$. Its action on $r_2$ takes it to 
\[
r_2+r_3-(r_2+r_3+r_4)= -r_4
\]
so $u_1$ conjugates $s_2$ to $s_4$ and vice versa. Finally, acting on the root $r_1$ takes it to
\[
r_1+(r_2+r_3+r_4)-(r_1+r_2+r_3+r_4+r_5)=-r_5,
\]
so conjugation by $u_1$ switches $s_1$ and $s_5$ as illustrated in the diagram. We note in passing that $-u_1$ is the so-called non-trivial pinned automorphism for $E_6$, as described by Reeder in \cite{reeder2012gradings}.

Turning to the element $u_2$, this is defined as the product of the four commuting reflections with roots 

\[
\{r_5,\,r_3,\, r_6+r_3+r_2,\, r_0+r_6+r_3+r_2+r_1\}.
\]
A similar calculation shows that conjugation by this element preserves $s_5, s_4, s_3$ and switches the pairs $s_6, s_2$ and $s_0, s_1$. Since the root representation is faithful we see that the product $u_1u_2$ has order three, and therefore these elements generate a copy of the dihedral group $\D_3$ and the element $u_3=u_1u_2u_1 = u_2u_1u_2$ is the product of four commuting reflections with roots 

\[
\{r_1,\,r_3,\, r_4+r_3+r_6,\,r_5+r_4+r_3+r_6+r_0\}.
\]

In many cases we will see that the $u_1,u_2$ dihedral group acts on one or more $2$-dimensional subtori of the maximal torus of $E_6$.  There are two natural actions of the dihedral group $\D_3$ on a $2$-torus: considering $\D_3$ as the Weyl group $W(A_2)$ we have the actions on the maximal tori in the Lie groups $SU_3$ and $PSU_3$.  The maximal tori are given respectively by triples $(\alpha,\beta,\gamma)\in \TT^3$ such that $\alpha\beta\gamma=1$, and the quotient of this by the group $C_3=\{(\omega,\omega,\omega): \omega^3=1\}$.  The action of the dihedral group simply permutes the three coordinates.  In both cases the maximal torus can be described as a hexagon with opposite sides identified, and the two actions correspond to the two conjugacy classes of $\D_3$ in $\D_6$. The orbifold quotients are an equilaterial triangle in the $SU_3$ case and a cone for $PSU_3$.

Given that $u_1,u_2$ act by signed permutations on the roots $r_1\dots,r_6,r_0$ we will see that the action of these on the $2$-tori are again given by signed permutations:
\begin{align*}
u_1: (\alpha,\beta,\gamma)\mapsto (\beta^{-1},\alpha^{-1},\gamma^{-1})\\
u_2: (\alpha,\beta,\gamma)\mapsto (\gamma^{-1},\beta^{-1},\alpha^{-1}).
\end{align*}
This has the effect of dualising the actions: the signed permutation action on the $SU_3$ torus is equivariantly isomorphic to the standard permutation action on the $PSU_3$ torus, and similarly the signed permutation action on the $PSU_3$ torus is equivariantly isomorphic to the standard permutation action on the $SU_3$ torus.

The two equivariant isomorphisms are both given by the map
\[
(\alpha:\beta:\gamma)\mapsto (\beta\gamma^{-1},\gamma\alpha^{-1},\alpha\beta^{-1}).
\]

\subsection{The elementary part of the centralisers} Our approach is slightly different to that of Carter, in that we begin with an elementary part of the centraliser that can easily be read off from the extended Coxeter-Dynkin diagram, and then seek additional elements to complete the centraliser.  Nonetheless our elementary part lies within Carter's direct product $W_1\times W_2$ for the conjugacy class representative.

Consider an element $w=w_1w_2\dots w_l \in W$ where $w_1,\dots w_l$ are root reflections and $l=\lbar(w)$ is the word length of an element $w\in W$ in terms of the generating set of all root reflections in the Weyl group.  Associated to such a word there is a root diagram (cf. Carter \cite{carter1972conjugacy}) where roots are connected by an edge when they are not orthogonal (here we only have single edges as we are in the $E_6$ case). We write our representatives in such a way that each connected component of the diagram corresponds to a subword. Such subwords $g_1,\dots,g_k$ of course commute with the element $w$ (and with one another) giving a direct product $\group{g_1}\times \dots \times \group{g_k}$ in the centraliser of $w$.  

Any roots which are orthogonal to all roots in the diagram will again give root reflections in the centraliser.  The group these generate is precisely Carter's group $W_2$. The elementary part of the centraliser is $\group{g_1}\times \dots \times \group{g_k}\times W_2$.

Carter determined the order of the centraliser subgroups, and for $11$ of the conjugacy classes this shows that the elementary part is the entire centraliser and the structure is obvious. In the remaining $14$ cases we are left to discover additional elements of the centraliser.  In all cases we will show that these additional generators can be found in the subgroup $\group{u_1,u_2,T,Ts_3}\cong \D_3\times \D_3$, as a consequence of our careful choice of representatives. 
In determining the centralisers and their actions we will also exploit the fact that the centraliser of an element $g$ is contained in the centraliser of $g^n$, and have chosen conjugacy class representatives accordingly. This relationship between the centralisers is summarised in Figure \ref{fig:centraliserinclusions}. 
Our choices of conjugacy class representatives are tabulated below. We list the eigenvalues for each of these representatives for the standard six dimensional representation: these are distinct, confirming that our representatives give distinct conjugacy classes and that the conjugacy classes are characterised by their eigenvalues.  We also describe the elementary part of the centraliser, and the index of this group in the full centraliser where it is not $1$.

\newgeometry{margin=0.3in}
\begin{landscape}
\begin{table}[ht]
\caption{The conjugacy class representatives and the elementary part of their centralisers}\label{Elementary centraliser table}
\begin{center}
\rowcolors{1}{}{lightgray}
\begin{tabular}{ccc|r|llc}
\hline
\rowcolor{white}\multicolumn{3}{c|}{\hbox{Conjugacy Class}}&{\hbox{centraliser}}&\multicolumn{3}{c}{\hbox{Elementary part of the centraliser}}\\
\rowcolor{white}\hbox{Carter Type}&\hbox{Representative}&\hbox{Eigenvalues}&\hbox{Order}&\hbox{Generators}&Structure&\hbox{Index in full centraliser}\\
\hline
$\emptyset$&$e$&$1,1,1,1,1,1$&$51840$&$\{e\}\times W$&$W$&\\
$A_1$&$s_0$&$1,1,1,1,1,-1$&$1440$&$\langle s_0\rangle\times \langle s_1, \ldots, s_5\rangle$&$ C_2\times S_6$&$ $\\
$A_1^2$&$ s_0s_1$&$1,1,1,1,-1,-1$&$192$&$\langle s_0,s_1\rangle\times \langle s_3,\ldots , s_5\rangle$&$ C_2\times C_2\times S_4$&$2$\\
$A_2$&$s_0s_6$&$1,1,1,1,\e23,\e43$&$216$&$\langle s_0s_6\rangle\times \langle s_1,s_2, s_4,s_5\rangle$&$ C_3\times S_3\times S_3$&$2$\\
$A_1^3$&$s_0 s_1s_5$&$1,1,1,-1,-1,-1 $&$96$&$\langle s_0, s_1, s_5\rangle\times \langle s_3\rangle$&$ C_2\times C_2\times C_2\times C_2$&$6$\\
$A_2\times A_1$&$s_0s_6s_1$&$1,1,1,\e23,\e43,-1$&$36$&$\langle s_0s_6, s_1\rangle\times \langle s_4,s_5\rangle$&$ C_3\times C_2\times S_3$&$ $\\
$A_3$&$s_0s_6s_3$&$1,1,1,i,-1,-i$&$32$&$\langle s_0s_6s_3\rangle\times \langle s_1,s_5\rangle$&$ C_4\times C_2\times C_2$&$2$\\
$A_1^4$&$s_0s_1s_5s_3$&$1,1,-1,-1,-1,-1$&$1152$&$\langle s_0,s_1,s_5,s_3\rangle$&$ C_2\times C_2\times C_2\times C_2$&$72$\\
$A_2\times A_1^2$&$s_0s_6s_1s_5$&$1,1,\e23,\e43,-1,-1$&$24$&$\langle s_0s_6,s_1,s_5\rangle$&$ C_3\times C_2\times C_2$&$2$\\
$A_2^2$&$s_0s_6s_1s_2$&$1,1,\e23,\e43,\e23,\e43$&$108$&$\langle s_0s_6,s_1s_2\rangle\times \langle s_4,s_5\rangle$&$ C_3\times C_3\times S_3$&$2$\\
$A_3\times A_1$&$s_0s_6s_3s_1$&$1,1,i,-1,-i,-1$&$16$&$\langle s_0s_6s_3, s_1\rangle \times \langle s_5\rangle$&$ C_4\times C_2\times C_2$&$ $\\
$A_4$&$s_0s_6s_3s_4$&$1,1,\e25,\e45,\e65,\e85$&$10$&$\langle s_0s_6s_3s_4\rangle\times \langle s_1\rangle$&$ C_5\times C_2$&$ $\\
$D_4$&$s_0s_1s_5T$&$1,1,-1,-1,\e13,\e53$&$36$&$\langle s_0s_1s_5T\rangle$&$ C_6$&$6$\\
$D_4[a_1]$&$s_1Ts_5s_0^T$&$1,1,i,-i,i,-i$&$96$&$\langle s_1Ts_5s_0^T\rangle$&$ C_4$&$24$\\
$A_2^2\times A_1$&$s_0s_6s_5s_1s_2$&$1,\e23,\e43,-1,\e23,\e43$&$36$&$\langle s_0s_6,s_5,s_1s_2\rangle$&$ C_3\times C_2\times C_3$&$2$\\
$A_3\times A_1^2$&$ s_0s_6s_3s_1s_5$&$1,i,-1,-i,-1,-1$&$96$&$\langle s_0s_6s_3,s_1,s_5\rangle$&$ C_4\times C_2\times C_2$&$6$\\
$A_4\times A_1$&$s_0s_6s_3s_4s_1$&$1,\e25,\e45,\e65,\e85,-1$&$10$&$\langle s_0s_6s_3s_4,s_1\rangle$&$ C_5\times C_2$&$ $\\
$A_5$&$s_0s_6s_3s_4s_5$&$1,\e13,\e23,-1,\e43,\e53$&$12$&$\langle s_0s_6s_3s_4s_5\rangle\times \langle s_1\rangle$&$ C_6\times C_2$&$ $\\
$D_5$&$s_0s_6s_3s_4s_3^{s_2s_4}$&$1,-1,\e14,\e34,\e54,\e74$&$8$&$\langle s_0s_6s_3s_4s_3^{s_2s_4}\rangle$&$ C_8$&$ $\\
$D_5[a_1]$&$s_0s_6s_3s_4T$&$1,-1,i,-i,\e13,\e53$&$12$&$\langle s_0s_6s_3s_4T\rangle$&$ C_{12}$&$ $\\
$A_2^3$&$s_0s_6s_1s_2s_5s_4$&$\e23,\e43,\e23,\e43,\e23,\e43$&$648$&$\langle s_0s_6,s_1s_2,s_4s_5\rangle$&$ C_3\times C_3\times C_3$&$24$\\
$A_5\times A_1$&$s_0s_6s_3s_4s_5s_1$&$\e13,\e23,-1,\e43,\e53,-1$&$36$&$\langle s_0s_6s_3s_4s_5,s_1\rangle$&$ C_6\times C_2$&$3$\\
$E_6$&$s_1s_2s_3s_4s_5s_6$&$\e23,\e43,\e16,\e56,\e76,\e{11}6$&$12$&$\langle s_1s_2s_3s_4s_5s_6\rangle$&$ C_{12}$&$ $\\
$E_6[a_1]$&$s_1s_2s_3s_4s_5s_6^{s_3}$&$\e29,\e49,\e89,\e{10}9,\e{14}9,\e{16}9$&$9$&$\langle s_1s_2s_3s_4s_5s_6^{s_3}\rangle$&$ C_9$&$ $\\
$E_6[a_2]$&$s_6s_2s_0^Ts_1^Ts_4s_3$&$\e13,\e13,\e23,\e43,\e53,\e53$&$72$&$\langle s_6s_2s_0^Ts_1^Ts_4s_3\rangle$&$ C_6$&$12$\\
\hline\end{tabular}
\end{center}
\end{table}
\end{landscape}
\restoregeometry

\subsection{The full centralisers}

We begin by stating an elementary lemma which will help to identify elements of the centraliser.

\begin{lemma}\label{eigenvectordecomp} Let $A,B$ be diagonalisable matrices and let $V_1\oplus\dots\oplus V_n$ be the eigenspace
decomposition for $B$. If for all but one $i$, there is an eigenspace
of $A$ containing $V_i$ then $A,B$ commute.
\end{lemma}

\begin{itemize}

\item In the cases $\emptyset,A_1,A_2\times A_1,A_3\times A_1,A_4,A_4\times A_1,A_5,D_5,D_5[a_1],E_6,E_6[a_1]$ the entire centraliser is the elementary part described in the previous section.

\item In cases $A_2,A_2\times A_1^2,A_3$, the representatives are $s_0s_6$, $s_0s_6s_1s_5$ and $s_0s_6s_3$ respectively, which are all centralised by the element $u_1$. In each case the elementary part of the centraliser has index $2$ and it is easy to see that this does not contain $u_1$. Moreover the action of the $u_1$ on the elementary part gives the structure of a wreath product as follows:
\[
\begin{array}{lllll}
Z_W(s_0s_6) &=& \langle s_0s_6\rangle \times\left(\langle s_1,s_2\rangle \wr \langle u_1\rangle \right) &\cong& C_3\times (S_3\wr C_2)\\
Z_W(s_0s_6s_1s_5) &=& \group{s_0s_6}\times (\langle s_1\rangle \wr\langle u_1)\rangle &\cong& C_3\times (C_2\wr C_2)\\
Z_W(s_0s_6s_3) &=&\group{s_0s_6s_3}\times (\langle s_1\rangle\wr \langle u_1\rangle) &\cong& C_4\times (C_2\wr C_2)
\end{array}
\]
The factors here are all complex reflection groups with the exception of $S_3 \wr C_2$.  This group is isomorphic to the affine orthogonal group $O^-_2(\F_3)\ltimes \F_3^2$ which can be embedded as a reflection group in $GL_3(\F_3)$, dual to the embedding as affine orthogonal transformations.  Note that this group is also isomorphic to $O^+_4(\F_2)$ which is unique amongst orthogonal groups as it is not generated by reflections of $\F_2^4$.

\item Similarly in cases $A_1^2,A_2^2,A_2^2\times A_1$, the representatives $s_0s_1,s_0s_6s_1s_2,s_0s_6s_5s_1s_2$ are centralised by $u_2$ and the structures are given by:
\[
\begin{array}{lllll}
Z_W(s_0s_1) &=& \langle s_0\rangle\wr \langle u_2\rangle \times \group{s_3, s_4, s_5}&\cong& (C_2\wr C_2)\times S_4
\\
Z_W(s_0s_6s_1s_2) &=& (\langle s_0s_6\rangle\wr\group{u_2}) \times\group{s_4,s_5} &\cong& (C_3\wr C_2) \times S_3
\\
Z_W(s_0s_6s_5s_1s_2) &=& (\group{s_0s_6}\wr\group{u_2})\times \group{s_5} &\cong& (C_3\wr C_2)\times C_2
\\
\end{array}
\]

\item In cases $A_1^3$ with representative $s_0s_1s_5$ and $D_4$ with representative $s_0s_1s_5T$, the elementary part of the centraliser has index $6$. The elements $s_0,s_1,s_5$ are permuted and $T$ is fixed by the group $\group{u_1,u_2}$, which therefore lies in both centralisers.

In the first case, since $\group{u_1,u_2}$ also commutes with $s_3$, the elementary part $\langle s_0\rangle \times\langle s_1\rangle\times \langle s_5\rangle\times \langle s_3\rangle$ is normal, and moreover faithfulness of the $\group{u_1,u_2}$ action implies these two groups have trivial intersection. Therefore
\[
Z_W(s_0s_1s_5)=(\langle s_0\rangle{\textstyle\wre{3}} \langle u_1, u_2\rangle)\times \langle s_3\rangle \cong
(C_2{\textstyle\wre{3}} \D_3)\times C_2
\]
where the notation $\wre{3}\D_3$ indicates the permutation wreath product for the action of $\D_3$ on the three factors.

For the element $s_0s_1s_5T$ the elementary part of the centraliser is simply the cyclic group $\group{s_0s_1s_5T}\cong C_6$, so its intersection with $\group{u_1,u_2}$ must be central and therefore trivial. Hence we obtain the centraliser as
\[
Z_W(s_0s_1s_5T)=\group{s_0s_1s_5T}\times\group{u_1, u_2}\cong C_6\times S_3.
\]

\item In the case $A_1^4$  with representative $s_0s_1s_5s_3$, the elementary part of the centraliser is the same as it was in the $A_1^3$ case above.  The  $-1$ eigenspace is the space spanned by $r_0, r_1, r_5, r_3$, and in particular it contains $r_T=1/2( r_0+ r_1+ r_5,+r_3)$. Hence the reflection $T$ lies in the centraliser. As noted in Section \ref{elements}, the group $\group{s_0, s_1, s_5, T}$ is  a Weyl group of type $D_4$, which has index $6$ in the centraliser. The group $\group{u_1,u_2}$ also centralises $s_0s_1s_5s_3$ and normalises $\group{s_0, s_1, s_5, T}$.

It remains to consider the intersection which, since $\group{u_1,u_2}$ centralises $T$, must lie in the centraliser $Z_{W(D_4)}(T)$ which has order $16$ (see \cite{carter1972conjugacy}). Explicitly this is the Weyl group of type $A_1^4$ with roots $r_0+r_1-r_T, r_0+r_5-r_T, r_1+r_5-r_T, r_T$, and the first three of these are permuted by $\group{u_1,u_2}$, hence the intersection is trivial.

We conclude that 
\[
Z_W(s_0s_1s_5s_3)=\group{s_0, s_1, s_5, T}\rtimes \group{u_1,u_2}\cong W(D_4)\rtimes S_3.
\]

\item Case $A_3\times A_1^2$, representative $s_0s_6s_3s_1s_5$: The elementary part of the centraliser is $\group{s_0s_6s_3}\times \langle s_1\rangle \times \langle s_5\rangle\cong C_4\times C_2\times C_2$. This can also be written as $\group{s_0s_6s_3s_1s_5}\times \langle s_1\rangle \times \langle s_5\rangle$, where the first factor is, tautologically, central in the full centraliser. Now we consider the  $-1$ eigenspace for our representative, which is $3$-dimensional. It  is spanned by the root vectors $r_1, r_5$ (the $-1$ eigenspaces of $s_1, s_5$ respectively) together with the $-1$ eigenspace for $s_0s_6s_3$, which is spanned by  $r_0+r_3$.  While $r_0+r_3$ is not a root vector, the sum $r_0+r_3+r_1+r_5$ is twice the root vector $r_T$, hence $T$ belongs to the centraliser. Considering the root system $\{r_1,r_T,r_5\}$ se see that the group $\group{s_1,T,s_5}\cong S_4$ and hence has trivial intersection with $\group{s_0s_6s_3s_1s_5}$. It follows that the centraliser is
\[
Z_W(s_0s_6s_3s_1s_5)=\group{s_0s_6s_3s_1s_5}\times \langle s_1, T, s_5\rangle\cong C_4\times S_4.
\]

\item Case $A_5\times A_1$, representative $s_0s_6s_3s_4s_5s_1$: The elementary part of the centraliser is the same as for the $A_5$ case, $\group{s_0s_6s_3s_4s_5}\times \langle s_1\rangle \cong C_6\times C_2$, and has index $3$ in the full centraliser. The $-1$ eigenspace of the representative is $2$ dimensional and is spanned by the $-1$ eigenvector $r_0+r_3+r_5$ of $s_0s_6s_3s_4s_5$ and the $-1$ eigenvector $r_1$ of $s_1$. As in the previous case $T$ is in the centraliser. We note that $\group{s_1,T}$ is a copy of $S_3$ which has trivial intersection with $\group{s_0s_6s_3s_4s_5s_1}$ since $\group{s_1,T}$ has trivial centre. Hence we obtain the centraliser subgroup as
\[
Z_W(s_0s_6s_3s_4s_5s_1)=\group{s_0s_6s_3s_4s_5s_1}\times \langle s_1,T\rangle \cong C_6\times S_3.
\]

\end{itemize}
The remaining cases are $E_6[a_2]$, $D_4[a_1]$ and $A_2^3$.  These all correspond to Springer-regular elements, \cite{springer1974regular}, and we can read off their structures from \cite[Table 1]{springer1974regular} along with the Shephard-Todd classification of complex reflection groups \cite[Table VII]{shephard1954finite} as the groups $G_5,G_8,G_{25}$ respectively.

In order to identify the actions on the fixed sets in the maximal torus, we need to identify not just the structure of these groups but their elements which we do as follows.

\begin{itemize}

\item Case $D_4[a_1]$, representative $s_1Ts_5s_0^T$: We choose this representative in the subgroup $\group{s_0, s_1, s_5, T}$ of type $D_4$,  so that  its square  $(s_1Ts_5s_0^T)^2=s_0s_1s_5s_3$ is the representative of the class of type $A_1^4$, and hence its centraliser is a subgroup of index 12 in the $A_1^4$ centraliser $\group{s_0, s_1, s_5, T}\rtimes \group{u_1,u_2}\cong W(D_4)\rtimes S_3$ (see above).  

We consider the standard $4$-dimensional representation of $W(D_4)$ corresponding to the generators $s_1,T,s_5, s_0$, with roots
\[
\rho_1:=(1,-1,0,0), \rho_T:= (0,1,-1,0), \rho_5:=(0,0,1,-1), \rho_0:=(0,0,1,1).
\]

The representative $s_1Ts_5s_0^T$ corresponds to the block diagonal matrix $R_{\pi/2}\oplus R_{\pi/2}$ which has $\pi/2$ rotations in coordinates $1,2$ and $3,4$. We equip $\R^4$ with the structure of a complex vector space, declaring that the operator $s_1Ts_5s_0^T$ is multiplication by $i$. Hence the elements of $W(D_4)$ centralising $s_1Ts_5s_0^T$ are precisely those which are $\mathbb C$-linear in this space. Specifically the intersection of our centraliser with the subgroup $W(D_4)$ is the Pauli group  $P$ generated by the matrices:
\[
\sigma_1=\left(\begin{array}{rr}0 & 1 \\1 & 0\end{array}\right),\qquad \sigma_2= \left(\begin{array}{rr}0& -i \\i & 0\end{array}\right),\qquad \sigma_3= \left(\begin{array}{rr}1 & 0 \\0 & -1\end{array}\right) .
\]

These correspond to the elements $s_1^Ts_5^T, T\left(s_1s_5s_0\right)^{Ts_5}$ and  $s_0s_5$  respectively.
The group $P=Z_W(s_1Ts_5s_0^T)\cap W(D_4)$ has index $6$ in the full centraliser, hence as the centraliser is contained in $W(D_4)\rtimes S_3$ it fits into an extension
\[
1\rightarrow P\rightarrow Z_W(s_1Ts_5s_0^T)\rightarrow S_3\rightarrow 1.
\]
We now identify lifts of the generators of $S_3$. The conjugation action of $u_3$ corresponds to interchanging $\rho_5, \rho_0$ and fixing $\rho_1, \rho_T$.
In particular, the conjugation action of $u_3$ on $R_{\pi/2}\oplus R_{\pi/2}$ inverts the second rotation, as does the element $s_5$, and therefore the element $s_5u_3$ is in the centraliser of $s_1Ts_5s_0^T$. It is represented as a complex $2\times 2$ matrix it is the diagonal matrix with entries $1,i$: we denote this complex matrix by $\alpha$.

Now we lift $u_1$, conjugation by which corresponds to exchanging $\rho_1, \rho_5$ while preserving $\rho_T, \rho_0$.
Conjugating $R_{\pi/2}\oplus R_{\pi/2}$ by this we obtain the matrix
\[
\left(\begin{array}{r|r} 0 & R_{\pi/2} \\\hline R_{\pi/2} & 0\end{array}\right)
\]
which can be conjugated back to $R_{\pi/2}\oplus R_{\pi/2}$ by the transposition $(2 4)$, given by $s_5^T$. Hence $s_5^Tu_1$ lies in the centraliser of $s_1Ts_5s_0^T$.  This is represented by the complex matrix
\[
\beta:=\frac 12 \left(\begin{array}{rr} 1-i & 1+i \\ 1+i & 1-i \end{array}\right).
\]
To summarise, the centraliser of $s_1Ts_5s_0^T$ is generated by the Pauli matrices $\sigma_1,\sigma_2,\sigma_3$ along with $\alpha,\beta$.

Note that the lifts $\alpha$ and $\beta$ each have order $4$, indeed $\alpha^2=\sigma_3$ and $\beta^2=\sigma_1$.  Moreover $\alpha\sigma_1\alpha^{-1}=\sigma_2$, hence the generators $\alpha,\beta$ are sufficient. The centraliser is thus the group
\[
\group{s_5u_3,s_5^Tu_1}
\]
which is faithfully represented as the matrix group $\group{\alpha,\beta}$.  These complex reflections  generate the symmetry group of the complex polygon $4(96)4$ (see section 1.6 of \cite{shephard1952regular}), identifying $Z_W(s_1Ts_5s_0^T)$ with the complex reflection group $G_8$ in the Shephard-Todd classification.

\item Case $A_2^3$, representative $s_0s_6s_1s_2s_5s_4$:
The elementary part of the centraliser is $\langle s_0s_6\rangle\times \langle s_1s_2\rangle\times\langle s_5s_4\rangle\cong C_3\times C_3\times C_3$. This has index $24$ in the full centraliser, and it is evident that the missing elements include $\group{u_1, u_2}$, which permute the direct factors of the elementary part, giving a subgroup of index $4$ in the centraliser isomorphic to $C_3\wre{3} \D_3$.

Applying Lemma \ref{eigenvectordecomp} we see that $Ts_3$ is in the centraliser. We will show that $\group{s_0s_6,Ts_3, s_5s_4}$ is the complex reflection group $G_{25}$ and hence is the full centraliser.

We start with the six dimensional real representation of the Weyl group $W(E_6)$ on the real vector space spanned by the $E_6$ root system.  On this space define the operator
\[
J:= \frac{1}{\sqrt{3}}\left( 2s_0s_6s_1s_2s_5s_4+I\right).
\]
 Note that  
 \[
 J^2= \frac{1}{3}\left(4 (s_0s_6s_1s_2s_5s_4)^2+4 s_0s_6s_1s_2s_5s_4 +I\right) = -I
 \]
 since $(s_0s_6s_1s_2s_5s_4)^2+ s_0s_6s_1s_2s_5s_4 + I=0$. Hence we can define a complex scalar multiplication on the Lie algebra using $J$ as multiplication by $i$. An element of the Weyl group $W(E_6)$ centralises $s_0s_6s_1s_2s_5s_4$ if and only if it commutes with $J$, i.e.\ is $\mathbb C$-linear for this structure.
 
 Moreover the elements $s_0s_6,Ts_3, s_5s_4$ are complex reflections: taking $r_0, r_1, r_5$ as a basis of this complex space these elements are given by:
 \[
 \alpha_1=\left(\begin{array}{rrr}\zeta & 0 & 0 \\0 & 1 & 0 \\0 & 0 & 1\end{array}\right), \qquad       \alpha_2=\frac{1}{\sqrt{3}} \left(\begin{array}{rrr}\eta & -\etabar & -\etabar \\-\etabar & \eta & -\etabar \\-\etabar & -\etabar & \eta\end{array}\right)        , \qquad \alpha_3=\left(\begin{array}{rrr}1 & 0 & 0 \\0 & 1 & 0 \\0 & 0 & \zeta\end{array}\right)
 \]
 where $\eta=e^{\pi i / 6}$. These generating reflections satisfy the  braid relations
 \[
 \alpha_1\alpha_2\alpha_1=\alpha_2\alpha_1\alpha_2, \alpha_2\alpha_3\alpha_2=\alpha_3\alpha_2\alpha_3,
 \]
and also $\alpha_1\alpha_3=\alpha_3\alpha_1$. Since the geometry of the fixed hyperplanes of the complex reflections is determined by the above relations, the reflection group is the Shephard-Todd group $G_{25}=\group{s_0s_6,Ts_3, s_5s_4}$.

\item Case $E_6[a_2]$. The representative 
$s_6s_2s_0^Ts_1^Ts_4s_3$ squares to the representative $s_0s_6s_1s_2s_5s_4$ above,  therefore its centraliser is a subgroup of $G_{25}$. In the above generators $\alpha_1, \alpha_2, \alpha_3$ we have
\[
s_6s_2s_0^Ts_1^Ts_4s_3 = \alpha_1^2\alpha_2^2\alpha_3\alpha_2^2\alpha_1^2 =
\left(\begin{array}{ccc}0 & -\zetabar & 0\\ -\zetabar & 0 & 0 \\0 & 0 & -\zetabar\end{array}\right).
\]
This evidently commutes with $\group{\alpha_2,\alpha_3}$ which is which is the binary tetrahedral group $SL_2(3)$, that is $G_4$ in the Shephard-Todd classification. Its centre is a copy of $C_2$ generated by the element $(\alpha_1^2\alpha_2^2\alpha_3\alpha_2^2\alpha_1^2)^3=(\alpha_3^2\alpha_2)^2$, hence $\group{(\alpha_1^2\alpha_2^2\alpha_3\alpha_2^2\alpha_1^2)^2}$ intersects trivially with $\group{\alpha_2,\alpha_3}$.

The group $C_3\times SL_2(3)$ is the Shephard-Todd group $G_5$, so we have the found the full centraliser
\[
Z_W(s_6s_2s_0^Ts_1^Ts_4s_3)=\group{(\alpha_1^2\alpha_2^2\alpha_3\alpha_2^2\alpha_1^2)^2}\times\group{\alpha_2, \alpha_3}\cong C_3\times SL_{2}(3).
\]

\end{itemize}

\section{Fixed sets and their quotients: the simply connected type}\label{fixed sets}

In this section we will determine the fixed sets of each conjugacy class representative in the standard maximal torus $\MaxTorus$ in the simply connected form of the Lie group of type $E_6$. We will analyse in detail the actions of the centralisers on the fixed sets for each conjugacy class representative.  In each case we will see that the quotient, which we refer to as a \emph{sector}, has the homotopy type of a finite union of points or a finite union of circles.  These results are summarised in Table \ref{fixed sets table}.

\bigskip

\begin{table}[ht]
\caption{Centralisers, Fixed Sets and Quotients}\label{fixed sets table}\label{Fixed sets and centralisers}

\begin{center}
\rowcolors{1}{}{lightgray}
\begin{tabular}{|ll|ll|lr|}
\hline
\rowcolor{white}\multicolumn{2}{|c|}{\hbox{Conjugacy class}}&\multicolumn{2}{|c|}{\hbox{Centraliser}}&\hbox{Fixed set}&\hbox{Quotient}\\
\rowcolor{white}\hbox{Type}&\hbox{Representative}&\hbox{Structure}&\hbox{Generators}&\hbox{}&\hbox{}\\
\hline
$\emptyset$&$e$&$W(E_6)$&$\gp{s_1, \dots, s_6}$&$\TT^6$&$\Delta^6$\\
$A_1$&$s_0$&$C_2\times W(A_5)$&$\gp{s_0}\times\gp{s_1, \dots, s_5}$&$\TT^5$&$\Delta^5$\\
$A_1^2$&$ s_0s_1$&$\D_4\times S_4$&$(\gp{s_0}\wr\gp{u_2})\times \gp{s_3,s_4,s_5}$&$\TT^4$&$\Delta_3\tilde\times S^1$\\
$A_2$&$s_0s_6$&$C_3\times(S_3\wr C_2)$&$\gp{s_0s_6}\times (\gp{s_1, s_2}\wr \gp{u_1})$&$\TT^4$&$\hbox{SP}^2(\Delta_2)$\\
$A_1^3$&$s_0 s_1s_5$&$(C_2\wr_3 \D_3)\times C_2$&$(\gp{s_0}\wr_3 \gp{u_1, u_2})\times \gp{s_3}$&$\TT^3$&$\Delta_2\times \Delta_1$\\
$A_2\times A_1$&$s_0s_6s_1$&$C_3\times C_2\times S_3$&$\gp{s_0s_6,s_1}\times \gp{s_4, s_5}$&$\TT^3$&$\TT^1\times \Delta_2$\\
$A_3$&$s_0s_6s_3$&$C_4\times \D_4$&$\gp{s_0s_6s_3}\times (\gp{s_1}\wr\gp{u_1})$&$\TT^3$&$\Delta_2\tilde\times S^1$\\
$A_1^4$&$s_0s_1s_5s_3$&$G_{28}$&$\gp{s_0, s_1, s_5, T}\rtimes \gp{u_1, u_2}$&$\TT^2\times V_4$&$\Delta_2\sqcup \Delta_2$\\
$A_2\times A_1^2$&$s_0s_6s_1s_5$&$C_3 \times \D_4$&$\gp{s_0s_6}\times (\gp{s_1}\wr\gp{u_1})$&$\TT^2$&$\hbox{SP}^2(\TT^1)$\\
$A_2^2$&$s_0s_6s_1s_2$&$(C_3\wr C_2)\times S_3$&$(\gp{s_0s_6}\wr\gp{u_2})\times\gp{s_4, s_5}$&$\TT^2\times \cZ$&$\Delta_2\times \cZ$\\
$A_3\times A_1$&$s_0s_6s_3s_1$&$C_4\times V_4$&$\gp{s_0s_6s_3}\times\gp{s_1}\times\gp{s_5}$&$\TT^2$&$\Delta_1\times \TT^1$\\
$A_4$&$s_0s_6s_3s_4$&$C_5\times C_2$&$\gp{s_0s_6s_3s_4}\times\gp{s_1}$&$\TT^2$&$\hbox{SP}^2(\TT^1)$\\
$D_4$&$s_0s_1s_5T$&$C_6\times \D_3$&$\gp{s_0s_1s_5T}\times\gp{u_1,u_2}$&$\TT^2$&$\Delta_2$\\
$D_4[a_1]$&$s_1Ts_5s_0^T$&$G_8$&$\gp{s_5u_3, s_5^Tu_1}$&$\TT^2$&$\Delta_2$\\
$A_2^2\times A_1$&$s_0s_6s_5s_1s_2$&$(C_3\wr C_2)\times C_2$&$(\gp{s_0s_6}\wr \gp{u_2})\times \gp{s_5}$&$\TT^1\times \cZ$&$\Delta_1\times \cZ$\\
$A_3\times A_1^2$&$ s_0s_6s_3s_1s_5$&$C_4\times S_4$&$\gp{\rep}\times\gp{s_1, T, s_5}$&$\TT^1\times V_4$&$\TT^1\sqcup \TT^1$\\
$A_4\times A_1$&$s_0s_6s_3s_4s_1$&$C_5\times C_2$&$\gp{s_0s_6s_3s_4s_5s_1}$&$\TT^1$&$\TT^1$\\
$A_5$&$s_0s_6s_3s_4s_5$&$C_6\times S_2$&$\gp{s_0s_6s_3s_4s_5}\times\gp{s_1}$&$\TT^1\times \cZ$&$\Delta_1\times \cZ$\\
$D_5$&$s_0s_6s_3s_4s_3^{s_2s_4}$&$C_8$&$\gp{s_0s_6s_3s_4s_3^{s_2s_4}}$&$\TT^1$&$\TT^1$\\
$D_5[a_1]$&$s_0s_6s_3s_4T$&$C_{12}$&$\gp{s_0s_6s_3s_4T}$&$\TT^1$&$\TT^1$\\
$A_2^3$&$s_0s_6s_1s_2s_5s_4$&$G_{25}$&$\gp{s_0s_6, Ts_3, s_5s_4}$&$C_3\times C_3\times \cZ$&Four points\\
$A_5\times A_1$&$s_0s_6s_3s_4s_5s_1$&$C_6\times S_3$&$\gp{\rep}\times \gp{s_1, T}$&$V_4\times \cZ$&$\cZ\sqcup \cZ$\\
$E_6$&$s_1s_2s_3s_4s_5s_6$&$C_{12}$&$\gp{s_1s_2s_3s_4s_5s_6}$&$\cZ$&$\cZ$\\
$E_6[a_1]$&$s_1s_2s_3s_4s_5s_6^{s_3}$&$C_9$&$\gp{s_1s_2s_3s_4s_5s_6^{s_3}}$&$\cZ$&$\cZ$\\
$E_6[a_2]$&$s_6s_2s_0^Ts_1^Ts_4s_3$&$G_5$&$\gp{(\rep)^2, Ts_3,s_5s_4}$&$\cZ$&$\cZ$\\
\hline\end{tabular}
\end{center}

\end{table}

Notation
\begin{itemize}
    \item $\TT^m$ the $m$-torus
    \item $\Delta_k$ the $k$-simplex (not necessarily equilateral)
    \item $\Delta_k\tilde\times S^1$ a twisted bundle over the circle with fibre a $k$-simplex
    \item $SP^2(X)$ the $2$-fold symmetric product of a space $X$
    \item $C_m$ the cyclic group of order $m$
    \item $\D_3$ the dihedral group of order $6$ generated by the special elements $u_1, u_2$
    \item $C_2\wr_3 \D_3$ is the permutation wreath product $C_2^3\rtimes \D_3$
    \item $\D_4$  a dihedral group of order 8 of the form $\langle s_i\rangle\wr \langle u_j\rangle$
    \item $S_m$ a symmetric group on $m$ letters generated by a subset of the 36 root reflections
    \item $V_4$ a Klein $4$ group.
    \item $G_5, G_8, G_{25}, G_{28}$ exceptional complex reflection groups (cf.\ \cite{shephard1954finite})
    \item $\rep$ is an abbreviation for the chosen representative of the conjugacy class.
     \item $\cZ$ the centre of the simply connected Lie group of type $E_6$.
\end{itemize}

\bigskip

The fixed sets are identified as topological spaces using Theorem \ref{Cardinality of component groups} combined with the following well known lemma. Recall that $\overline{l}(w)$ denotes the length of the element $w$ with respect to the generating set of \emph{all} root reflections.

\begin{lemma}[{\cite[Lemma 2]{carter1972conjugacy}}]
$\overline{l}(w)$ is the number of eigenvalues with multiplicity of $w$ which are not equal to $1$.
\end{lemma}

Using this lemma, for each $w\in W$ the torus $\MaxTorus^w_1$ has dimension ${6-\overline{l}(w)}$, where $\overline{l}(w)$ denotes the word length of $w$ with respect to the generating set consisting of all root reflections in $W$. 

Combining this with Theorem \ref{homeomorphic fixed sets} determines the topology of the fixed sets in both the simply connected and adjoint forms:

\begin{theorem}
\label{the fixed sets}
For each $w\in W$ the fixed sets $\MaxTorus^w$ and $(\MaxTorus^\vee)^w$ are both homeomorphic to the disjoint union of $g$ copies of a $6-\overline{l}(w)$-torus where $g$ is the greatest common divisor of the $r\times r$-minors of the matrix of $I-w$ and $r$ is the rank of $I-w$.
\end{theorem}

We now turn to the actions of centralisers on the corresponding fixed sets.

Since the action of the Weyl group $W$ on the maximal torus $\MaxTorus$ is given by conjugation, all of the fixed sets contain the centre $\cZ$ of the Lie group as a pointwise fixed subset. For example in the 3 elliptic cases of Carter-type $E_6,E_6[a_1],E_6[a_2]$, the fixed sets, and hence the quotients are identified with the centre $\cZ$.

In cases $A_4\times A_1, D_5$ and $D_5[a_1]$ the fixed set consists of a single circle containing the (pointwise fixed) centre so any element of $W$ preserving the circle must fix it pointwise.  In particular the centralisers pointwise fix the circle identifying the quotient with the fixed set.

\subsection{Alternative Coordinate Systems}
To understand the actions of centralisers on the fixed sets it will be helpful to consider alternative coordinate systems on the torus.

\begin{enumerate}
\item $\mathfrak{su}_6$ coordinates: In many cases the fixed set will lie inside the $5$-torus $(\alpha,\beta,\gamma,\delta,\epsilon,1)$ which is the maximal torus corresponding to the $A_5$ root system given by $\{r_1,\dots ,r_5\}$.  When considering the $A_n$ case it is often helpful to use coordinates where the roots are $(1,-1,0,\dots,0)$, $(0,1,-1,0,\dots,0)$ etc.\ and for which the maximal torus is correspondingly given by tuples with product $1$ (as is the case for the standard maximal torus in $SU_{n+1}$).  We obtain these coordinates on the above $5$-torus by the change of coordinates:
\[(\alpha,\beta,\gamma,\delta,\epsilon,1)\mapsto (\alpha, \beta\alpha^{-1}, \gamma\beta^{-1}, \delta\gamma^{-1},\epsilon\delta^{-1}, \epsilon^{-1}).
\]
Under this identification the $A_5$ Weyl group generated by $s_1,\ldots,s_5$ acts by permuting the coordinates.

\item $\mathfrak{su}_3^3$ coordinates: Another alternative coordinate system is obtained by considering the roots $r_1,r_2,r_5,r_4,r_0,r_6$ as a basis for the Lie algebra.  The advantage of this approach is that the dihedral group $\group{u_1,u_2}$ acts by signed permutations on the pairs $(r_1,r_2), (r_5,r_4)$ and $(r_0,r_6)$.  Moreover the actions of $s_1,s_5,s_0$ will remain relatively simple in our new coordinates. The omitted root vector $r_3$ can be obtained from
\begin{align*}
3r_3 &= -r_0 -2r_6 -r_1-2r_2 -r_5-2r_4\\
&=-r_0 +r_6 -r_1+r_2 -r_5+r_4-3(r_6+r_2+r_4)
\end{align*}
hence the root lattice is the refinement of the lattice generated by $r_1,r_2,r_5,r_4,r_0,r_6$ by the addition of the vector  $\frac 13(-r_0 +r_6 -r_1+r_2 -r_5+r_4)$.

Motivated by this we take the ordered basis $r_1,-r_2,r_5,-r_4,r_0,-r_6$ for the Lie algebra, for which coordinates the lattice is $\Z^6$ extended by the vector $(\frac 13,\dots,\frac 13)$. The maximal torus is given by $\TT^6/C_3$ where the $C_3$ acts by multiplication by $(\omega,\dots,\omega)$ for $\omega$ a cube root of unity.  We denote these new coordinates $(\alpha':\beta':\epsilon':\delta':\eta':\zeta')$.

The action of $\group{u_1,u_2}$ on $\TT^6/C_3$ is to permute and invert pairs of coordinates as follows.
\begin{align*}
u_1:(\alpha':\beta':\epsilon':\delta':\eta':\zeta')&\mapsto((\epsilon')^{-1}:(\delta')^{-1}:(\alpha')^{-1}:(\beta')^{-1}:(\eta')^{-1}:(\zeta')^{-1})\\
u_2:(\alpha':\beta':\epsilon':\delta':\eta':\zeta')&\mapsto((\eta')^{-1}:(\zeta')^{-1}:(\epsilon')^{-1}:(\delta')^{-1}:(\alpha')^{-1}:(\beta')^{-1})\\
u_3:(\alpha':\beta':\epsilon':\delta':\eta':\zeta')&\mapsto ((\alpha')^{-1}:(\beta')^{-1}:(\eta')^{-1}:(\zeta')^{-1}:(\epsilon')^{-1}:(\delta')^{-1})
\end{align*}

The conversion from these coordinates to the standard $r_1,\dots,r_6$ coordinates is given by the map
\[
(\alpha':\beta':\epsilon':\delta':\eta':\zeta')\mapsto (\eta'(\alpha')^{-1},(\eta')^2\beta',(\eta')^3,(\eta')^2\delta',\eta'(\epsilon')^{-1},(\eta')^2\zeta')
\]
which is well defined on $C_3$-cosets.

Now applying $s_1$ changes the first of these coordinates to 
\[
(\eta'(\alpha')^{-1})^{-1}(\eta')^2\beta = \eta'\alpha'\beta'
\]
while leaving the other coordinates unchanged.  Thus in the new coordinates $s_1$ takes $(\alpha':\beta':\epsilon':\delta':\eta':\zeta')$ to $((\alpha'\beta')^{-1}:\beta':\epsilon':\delta':\eta':\zeta')$.

Similarly $s_5$ changes $\epsilon'$ to $(\epsilon'\delta')^{-1}$ leaving the other coordinates unchanged, and $s_0$ changes $\eta'$ to $(\eta'\zeta')^{-1}$ leaving the other coordinates unchanged.

The action of $T$ in the new coordinates is much simpler than in the original coordinates.  In standard coordinates $T$ is given by the matrix:
\[
\begin{pmatrix}
1&0&0&0&0&0\\
1&0&1&-1&1&-1\\
1&-1&2&-1&1&-1\\
1&-1&1&0&1&-1\\
0&0&0&0&1&0\\
1&-1&1&-1&1&0
\end{pmatrix}
\]
Applying this to $(\eta'(\alpha')^{-1},(\eta')^2\beta',(\eta')^3,(\eta')^3\delta',\eta'(\epsilon')^{-1},(\eta')^2\zeta)$, and choosing $\phi$ to be a cube root of $(\alpha'\beta'\epsilon'\delta'\zeta'\eta')^{-1}$ we obtain
\[
(\eta'(\alpha')^{-1},\eta'(\alpha'\epsilon'\delta'\zeta')^{-1},(\eta')^2(\alpha'\beta'\epsilon'\delta'\zeta')^{-1},\eta'(\alpha'\beta'\epsilon'\zeta')^{-1},\eta'(\epsilon')^{-1},\eta'(\alpha'\beta'\epsilon'\delta')^{-1})
\]
\[
= (\eta'(\alpha')^{-1},(\eta')^2\beta'\phi^3,(\eta')^3\phi^3,(\eta')^2\delta'\phi^3,\eta'(\epsilon')^{-1},(\eta')^2\zeta'\phi^3)
\]
\[
= ((\eta'\phi)(\alpha'\phi)^{-1},(\eta'\phi)^2\beta'\phi,(\eta'\phi)^3,(\eta'\phi)^2\delta'\phi,(\eta'\phi)(\epsilon'\phi)^{-1},(\eta'\phi)^2\zeta'\phi)
\]
whence the action of $T$ in the new coordinates simply multiplies each coordinate by $\phi$.  Since $\phi$ was well-defined up to a cube root of unity this makes the action well-defined on $\TT^6/C_3$.

\end{enumerate}
 
\subsection{The fixed sets and their quotients}

We now proceed to compute the fixed sets and sectors.

\begin{enumerate}
\item Case $\emptyset$:
The fixed set is $\MaxTorus$ and the centraliser is $W(E_6)$, which acts on $\MaxTorus$ with quotient a 6-simplex.

\medskip
\item Case $A_1$:
The fixed variety is a single $5$-torus $(\alpha,\beta,\gamma,\delta,\epsilon,1)$. This is canonically identified with the maximal torus for the Lie group of type $A_5$. The centraliser of $s_0$ is the direct product of $\group{s_0}$ acting trivially with the Weyl group $\group{s_1,s_2,s_3,s_4,s_5}$ of type $A_5$ acting on the $A_5$ maximal torus. The quotient of the fixed set by the centraliser is a 5-simplex.

\medskip\item Case $A_1^2$:
The fixed variety is a single $4$-torus $(\alpha,\alpha^2,\gamma,\delta,\epsilon,1)$ inside the $5$-torus where $\zeta=1$, hence we can use $\mathfrak{su}_6$ coordinates as above.

The fixed set of our element is then identified with the subset of the $SU_6$-torus
\begin{equation}\label{4-torus}
\{(\alpha, \alpha, \gamma'', \delta'', \epsilon'',\zeta''): \alpha, \gamma'', \delta'', \epsilon'', \zeta''\in \Ti\hbox{ and } \alpha^2\gamma''\delta''\epsilon''\zeta''=1\}.
\end{equation}

 The centraliser is the direct product of $\langle s_0\rangle\wr \langle u_2\rangle$ with the symmetric group $S_4=\group{s_3,s_4,s_5}$, where $s_0$ acts trivially and $S_4$ acts by permutation of the last four coordinates. First we factor out the action of the $S_4$ which we will do in logarithmic coordinates $(x_1, x_1, x_3, \ldots, x_6)$ subject to the constraint $x_1+x_1+x_3+x_4+x_5 +x_6 =0$. We take $x_1$ to vary from $-\sfrac12$ to $\sfrac12$.
 
 In these coordinates the elements of the $S_4$ factor act by permuting the coordinates $x_3, \ldots, x_6$, and have fundamental domain $D$ given by

 \begin{align*}
 D=\{(x_1, x_1,  x_3, \ldots x_6)\mid x_1\in [-\sfrac12, \sfrac12],  2x_1+x_3+x_4+x_5 +x_6 &=0,\\ x_3\leq x_4\leq x_5\leq x_6&\leq x_3+1\}.
 \end{align*}
 See \cite[Equation 11]{niblo2019stratified} for the details. 
 
 As noted in Example 7.2 \emph{ibid} the quotient of the $4$-torus (\ref{4-torus}) under the action of this $S_4$ is an orientable twisted $3$-simplex bundle over the circle parameterised by $\alpha$ with  monodromy defined by 
 \[
 (-\sfrac12, -\sfrac12, x_3, x_4, x_5, x_6)\mapsto (\sfrac12, \sfrac12, x_5-1, x_6-1, x_3, x_4).
 \]
 Now we consider the action of the element $u_2$. Since this commutes with the $S_4$ we can act directly on the quotient bundle described above.
Recall that $u_2$ negates $r_3, r_4, r_5$ and exchanges $r_0$ with $-r_1$ and $r_2$ with $-r_6$, so the action of $u_2$ on the fixed set is:
 \[
 (\alpha,\alpha^2,\gamma,\delta,\epsilon,1)\mapsto (\alpha,\alpha^{2},\alpha^{3}\gamma^{-1},\alpha^{2}\delta^{-1},\alpha\epsilon^{-1},1).
 \]
 Changing coordinates as above this becomes
 \[
 (\alpha, \alpha, \gamma'', \delta'', \epsilon'',\zeta'')\mapsto (\alpha, \alpha, (\alpha\gamma'')^{-1}, (\alpha\delta'')^{-1}, (\alpha\epsilon'')^{-1},(\alpha\zeta'')^{-1}).
 \]
 In logarithmic coordinates this is
 \[
 (x_1, x_1,  x_3, x_4, x_5, x_6)\mapsto (x_1, x_1, -x_1-x_3, -x_1-x_4, -x_1-x_5, -x_1-x_6).
 \]
 While this does not preserve the fundamental domain $D$ for the $S_4$ action, composing with the permutation that reverses the last 4 coordinates returns us to $D$, so we have the map 
 \[
 (x_1, x_1,  x_3, x_4, x_5, x_6)\mapsto (x_1, x_1, -x_1-x_6, -x_1-x_5, -x_1-x_4, -x_1-x_3).
 \]
 
 In the $x_1=0$ fibre the vertex set of the simplex is
\begin{align*}
v_0(0)&=(0,0,0,0,0,0),\\
v_1(0)&=(0,0,-\sfrac14,-\sfrac14,-\sfrac14,\sfrac34),\\
v_2(0)&=(0,0,-\sfrac12,-\sfrac12,\sfrac12,\sfrac12),\\
v_3(0)&=(0,0,-\sfrac34,\sfrac14,\sfrac14,\sfrac14)
 \end{align*}
 and in a general fibre the vertex set $\{v_i(x_1)\}$ is given by translating each of these vectors by $(x_1,x_1,-\frac {x_1}2,-\frac {x_1}2,-\frac {x_1}2,-\frac {x_1}2)$
  
 In each simplex two of the vertices are fixed by the $u_2$ action, namely $v_0(x_1)$ and $v_2(x_1)$ while the other two are interchanged. We think of the fibres as given by the join of two $1$-simplices corresponding to the edge $e_0=[v_0(x_1),v_2(x_1)]$ and the edge  $e_1=[v_1(x_1),v_3(x_1)]$.
 
The action of $u_2$ on the simplex bundle preserves the fibres and moreover preserves the join structure on each, yielding again a simplex bundle with fibres given by the join of $e_0$ with the quotient of $e_1$ by inversion.

 The quotient of the fixed set by the centraliser is therefore a simplex bundle over the circle, where the monodromy is given by the reflection of the $3$-simplex which inverts the $e_0$ factor and fixes the $e_1$ factor.  This yields a Seifert fibred $4$-orbifold: the base space is the quotient of the join of $e_0$ and $e_1$ by the Klein $4$ group inverting both of these edges,
 \cite{scott1983geometries}.

\medskip\item Case $A_2$:
The fixed variety is the  $4$-torus $(\alpha,\beta,1,\delta,\epsilon,1)$ inside the standard $5$-torus from case $2$. We again make use of $\mathfrak{su}_6$ coordinates. Under this identification the $A_5$ Weyl group $\group{s_1,\dots,s_5}$ acts by permuting the coordinates. The fixed set of our element is then identified with the subset $\{(\alpha, \beta\alpha^{-1}, \beta^{-1}, \delta,\epsilon\delta^{-1}, \epsilon^{-1}): \alpha, \beta, \delta, \epsilon\in \Ti\}$. Since the first three coordinates and the last three coordinates each multiply to $1$ this is naturally a product of two maximal tori of type $A_2$.

The centraliser  has the form $\langle s_0s_6\rangle \times\left(\langle s_1,s_2\rangle \wr \langle u_1\rangle \right)\cong C_3\times (S_3\wr C_2)$ with the element $s_0s_6$ acting trivially.  The subgroups $\langle s_1, s_2\rangle$ and $\langle s_5, s_4\rangle = \langle s_1, s_2\rangle^{u_1}$ act as the Weyl groups on these 2-torus factors so the direct product $S_3\times S_3$ acts in the natural way on $\TT^2\times \TT^2$ with quotient a product of two equilateral triangles. The element $u_1$ acts to swap these, so the quotient is the symmetric product of two copies of an equilateral triangle.

\medskip\item Case $A_1^3$:
The fixed set is a single $3$-torus $(\alpha, \alpha^2, \gamma, \epsilon^2, \epsilon, 1)$, and is the intersection of the fixed sets of each of the elements $s_0, s_1, s_5$. The centraliser is the direct product of $\group{s_3}$ with the wreath product $\group{s_0}\wr \group{u_1, u_2}$, where the conjugates of $\group{s_0}$ are $\group{s_1}, \group{s_5}$. Since these act trivially the centraliser acts as $\group{u_1, u_2}\times \group{s_3}$. 

Recall that $u_1$ exchanges $r_1$ with $-r_5$, $r_2$ with $-r_4$ and negates $r_3,r_6$, so on the fixed set the action of $u_1$ exchanges the values of $\alpha$ and $\epsilon^{-1}$, and inverts $\gamma$.  On the other hand $u_2$ is given by 
 \[
 (\alpha,\alpha^2,\gamma,\epsilon^2,\epsilon,1)\mapsto (\alpha,\alpha^{2},\alpha^{3}\gamma^{-1},\alpha^{2}\epsilon^{-2},\alpha\epsilon^{-1},1),
 \]
 see $A_1^2$ case.
 We decompose our $3$-torus as the product 
\[
\TT^2\times \TT^1=\{(\alpha, \alpha^2, \alpha\epsilon, \epsilon^2,  \epsilon, 1)\}\times \{(1,1,\gamma, 1,1,1)\}.
\]
and note that both factors are preserved by the actions of $u_1,u_2$.  Moreover $s_3$ pointwise fixes the first factor and inverts the second.  Since $u_1,u_2$ also invert the second factor, the products $u_1s_3, u_2s_3$ (which also generate a dihedral group) pointwise fix the second factor, and act on the first as
\begin{align*}
u_1s_3:& (\alpha,\epsilon)\mapsto (\epsilon^{-1},\alpha^{-1})\\
u_2s_3:& (\alpha,\epsilon)\mapsto (\alpha,\alpha\epsilon^{-1})
\end{align*}
which is the standard dihedral action with quotient an equilateral triangle.

Hence the quotient of the fixed set of $s_0s_1s_5$ by the centraliser is
\[
\TT^2/\group{u_1s_3,u_2s_3}\times \TT^1/\group{s_3} = \Delta^2\times I.
\]

\medskip\item Case $A_2\times A_1$:
The  fixed set is the single $3$-torus $(\alpha, \alpha^2, 1, \delta, \epsilon, 1)$ and the centraliser is $\group{s_0s_6}\times \langle s_1\rangle \times \langle s_4,s_5\rangle\cong C_3\times C_2\times S_3$. The first two factors act trivially so we are left with a residual action of $S_3=\group{s_4, s_5}$.  These generators act as the matrices $\left(\begin{array}{cc}-1 & 1 \\0 & 1\end{array}\right), \left(\begin{array}{cc}1 & 0 \\1 & -1\end{array}\right)$ on the free parameters $(\delta, \epsilon)$, fixing the parameter $\alpha$. The non trivial  part of the action is therefore the standard dihedral action on a $2$-torus  with quotient an equilateral triangle, crossed with the trivial action on the remaining circle. Hence the quotient is $\TT^1\times \Delta^2$.

\medskip\item Case $A_3$:
The fixed set is a single $3$-torus $(\alpha,\beta,1,\beta^{-1}, \epsilon, 1)$ and the centraliser of $s_0s_6s_3$  has the structure $\group{s_0s_6s_3}\times (\langle s_1\rangle\wr \langle u_1\rangle)\cong C_4\times (C_2\wr C_2)$, so we need to compute the quotient by the action of the wreath product $\langle s_1\rangle\wr \langle u_1\rangle$.

We begin by identifying the quotient under the action of $s_1$ and $s_1^{u_1}=s_5$.  The elements $s_1$ and $s_5$ both act trivially on the $\beta$ factor of the $3$-torus, with $s_1$ acting on the $\alpha, \beta$ subtorus by the matrix $\left(\begin{array}{cc}-1 & 1 \\0 & 1\end{array}\right)$ and $s_5$ acting on the $\beta, \epsilon$ subtorus via $\left(\begin{array}{cc}1 & 0 \\-1 & -1\end{array}\right)$.

We introduce logarithmic coordinates for $\beta$, i.e.\ set $\beta=e^{2\pi i x}$.  The 3-torus is then
$$\{(\alpha,x,\epsilon): \alpha,\epsilon\in \Ti, x\in [0,1]\}/ (\alpha,1,\epsilon)\sim (\alpha,0, \epsilon).$$
Now introduce coordinates: $\alpha' = \alpha e^{-\pi i x}$, and $\epsilon' = \epsilon e^{\pi i x}$. Under this change of coordinates $ (\alpha,1, \epsilon)$ goes to $ (\alpha',1, \epsilon')=(-\alpha, 1 ,-\epsilon)$ so in these coordinates the torus is:
$$\{(\alpha',x,\epsilon'): \alpha',\epsilon'\in \Ti, x\in [0,1]\}/ (\alpha',1,\epsilon')\sim (-\alpha',0, -\epsilon').$$
The action of $s_1$ fixes $\beta=e^{2\pi i x}$ and $\epsilon$, and hence also fixes $\epsilon'$, while
$$\alpha \mapsto \alpha^{-1}\beta = (\alpha'e^{\pi i x})^{-1}e^{2\pi i x} = (\alpha')^{-1}e^{\pi i x}.$$
Hence in primed coordinates, $s_1$ takes $\alpha'$ to $(\alpha')^{-1}e^{\pi i x}e^{-\pi i x}=(\alpha')^{-1}$.

Similarly $s_5$ fixes $\alpha'$ and $x$ and takes $\epsilon$ to $\epsilon^{-1}\beta^{-1}=(\epsilon')^{-1}e^{-\pi i x}$.  Hence $\epsilon'\mapsto (\epsilon')^{-1}$.

Setting $I=\{\phi\in \Ti : \mathrm{Im}\, \phi\geq 0\}$ which we identify in the natural way as the quotient of $\Ti$ by the map $\phi\mapsto \phi^{-1}$, the quotient of the torus by $\group{s_1,s_5}$ is given by
$$\{(\alpha',x,\epsilon'): \alpha',\epsilon'\in I, x\in [0,1]\}/ (\alpha',1,\epsilon')\sim ((-\alpha')^{-1},0, (-\epsilon')^{-1}).$$
This is a square bundle over the circle, whose monodromy map is rotation by $\pi$.

Now the action of $u_1$ preserves $\beta$ (and hence $x$), and exchanges $\alpha$ with $\epsilon^{-1}$. In primed coordinates
\begin{align*}
u_1: (\alpha',x,\epsilon')&\mapsto ((\epsilon')^{-1},x,(\alpha')^{-1})
\end{align*}
so descending to the square bundle the action simply exchanges $\alpha'$ with $\epsilon'$.

The quotient of each fibre is a right-isosceles triangle, hence the quotient of the fixed set by the centraliser is a triangle bundle over the circle with monodromy map reflecting the triangle. This is a Seifert bundle over the $(4,4,2)$ triangle, \cite{scott1983geometries}.

\medskip\item Case $A_1^4$:
The fixed set is the product of a $2$-torus indexed by $\alpha,\epsilon$ with a Klein 4 group generated by independent square-roots of $1$ denoted $\eta = \gamma(\zeta\alpha\epsilon)^{-1}$ and $\zeta$.
This gives coordinates $(\alpha, \zeta\alpha^2, \eta\zeta\alpha\epsilon, \zeta\epsilon^2, \epsilon, \zeta)$.

The centraliser is
\[
\group{s_0, s_1, s_5, T}\rtimes \group{u_1,u_2}\cong W(D_4)\rtimes S_3.
\]
We will now change to $\mathfrak{su}_3^3$ coordinates in which the actions of the generators of the centraliser are all relatively straightforward.

Recall that the tuple $(\alpha':\beta':\epsilon':\delta':\eta':\zeta')$ in new coordinates corresponds to $(\eta'(\alpha')^{-1},(\eta')^2\beta',(\eta')^3,(\eta')^2\delta',\eta'(\epsilon')^{-1},(\eta')^2\zeta')$ in the original coordinates so in the new coordinates a point of the fixed set must satisfy
\begin{align*}
(\eta')^2\beta'&=\zeta(\eta'(\alpha')^{-1})^2\\
(\eta')^3&=\eta\zeta(\eta'(\alpha')^{-1})(\eta'(\epsilon')^{-1})\\
(\eta')^2\delta'&=\zeta(\eta'(\epsilon')^{-1})^2\\
(\eta')^2\zeta'&=\zeta
\end{align*}
where $\eta,\zeta$ are $\pm{1}$ giving the element of the Klein $4$ group and the other variables lie in the circle. This simplifies to
\begin{align*}
\beta'&=\zeta(\alpha')^{-2}\\
\alpha'\epsilon'\eta'&=\zeta\eta\\
\delta'&=\zeta(\epsilon')^{-2}\\
\zeta'&=\zeta(\eta')^{-2}
\end{align*}
so for $(\eta,\zeta)$ in the Klein $4$ group $V_4$, the corresponding component of the fixed set is parametrised as
\[
\{(\alpha',\epsilon',\zeta') : \alpha'\epsilon'\eta'=\zeta\eta \}/\group{(\omega,\omega,\omega)}
\]
where $\omega$ is a non-trivial cube root of $1$.

In these coordinates the subgroup $\group{u_1,u_2}$ simply acts by the signed permutations:
\begin{align*}
u_1&: (\alpha':\epsilon':\zeta')\mapsto((\epsilon')^{-1}:(\alpha')^{-1}:(\zeta')^{-1})\\
u_2&: (\alpha':\epsilon':\zeta')\mapsto((\zeta')^{-1}:(\epsilon')^{-1}:(\alpha')^{-1})
\end{align*}
and preserve $\eta,\zeta$. As noted in Section \ref{section:special} this action on the $PSU_3$ torus is equivariantly isomorphic to the permutation action on the $SU_3$ torus, for which the quotient is an equilateral triangle.

The element $T$ acts by multiplication by $\phi$ where $\phi$ is a cube root of
\[(\alpha'\zeta(\alpha')^{-2}\epsilon'\zeta(\epsilon')^{-2}\eta'\zeta(\eta')^{-2})^{-1}
=\alpha'\epsilon'\eta'\zeta = \eta.
\]
Since $\eta=\pm1$ the action simply multiplies the projective coordinates  $\alpha',\epsilon',\eta'$ by $\eta$ and hence as $\alpha'\epsilon'\eta'=\zeta\eta$ the coordinate $\zeta$ is also multiplied by $\eta$.

We now turn to the actions of $s_1,s_5$ and $s_0$.  It suffices to consider the element $s_1$ as the group $\group{u_1,u_2}$ conjugates $s_1,s_5$ and $s_0$. For a general point we have:
\[
s_1:(\alpha':\beta':\epsilon':\delta':\eta':\zeta')\to((\alpha'\beta')^{-1}:\beta':\epsilon':\delta':\eta':\zeta')
\]
so applying this to an element $(\alpha':\zeta(\alpha')^{-2}: \epsilon':\zeta(\epsilon')^{-2}:\eta':\zeta(\eta')^{-2})$ we obtain
\[
(\zeta\alpha':\zeta(\zeta\alpha')^{-2}: \epsilon':\zeta(\epsilon')^{-2}:\eta':\zeta(\eta')^{-2})
\]
Since the product of the coordinates $\alpha'\epsilon'\zeta'$ changes by a factor of $\zeta$, this means that $\eta$ is changed by $\zeta$.

In the identity component $\eta=\zeta=1$, so $T,s_1$ (and hence also $s_5,s_0$) act trivially.   Hence $\MaxTorus^w_1/Z_W(w)$ is the quotient of the $PSU_3$ torus
\[
\{(\alpha',\epsilon',\zeta') :  \alpha'\epsilon'\zeta'=1\}/C_3
\]
by the signed permutation action of $\group{u_1,u_2}$.  As noted above this can be identified with the quotient of the $SU_3$ torus by its Weyl group, giving an equilateral triangle.

Turning to the components indexed by the non-identity elements of the Klein $4$ group we see that these are permuted transitively by $\group{s_1,T}$, hence there is just one more component in the quotient. We obtain this by taking the quotient of the $(\eta,\zeta)=(1,-1)$ component by its stabiliser in the centraliser, noting that this component is fixed pointwise by $T$.  This component is parametrised by
\[\{(\alpha',\epsilon',\zeta') : \alpha'\epsilon'\zeta'=-1\}/C_3,
\]
and we now determine its stabiliser.

As noted above the elements $u_1,u_2$ preserve all the components so we begin by considering the action of the subgroup $W(D_4)$ on the set of the three non-identity components. This action gives a map from $W(D_4)=\group{s_1,s_5,s_0,T}$ to the permutation group $S_3$ in which $s_1,s_5,s_0$ all map to one transposition and $T$ maps to another.  Thus the component stabiliser in $W(D_4)$ is generated by $T$ along with the kernel of this map.  Clearly the elements $s_1s_5$, $s_0s_5$, $(s_1s_5)^T$ and $(s_0s_5)^T$ lie in the kernel. We note that $s_1s_5(s_0s_5)^T$ has order $4$ (it corresponds to a pair of rotations by $\pi/2$ in the standard representation of $W(D_4)$) hence $\group{s_1s_5,(s_0s_5)^T}$ is a dihedral group of order $8$, and similarly for $\group{s_0s_5,(s_1s_5)^T}$.  All the other pairs from  $s_1s_5$, $s_0s_5$, $(s_1s_5)^T$ and $(s_0s_5)^T$ have product of order $2$ and hence pairwise commute.  The intersection of the two dihedral groups is precisely the centre of each group (which has order 2) and hence together they generate a group of order $32$ which is thus the whole of the kernel.

Hence the component stabiliser in $W(D_4)$ is the semidirect product
\[\group{s_1s_5,s_0s_5, (s_1s_5)^T,(s_0s_5)^T}\rtimes \group{T}=\group{s_1s_5,s_0s_5,T}.\]
As noted above the action of $T$ on the  $(\eta,\zeta)=(1,-1)$ component  is trivial. On the other hand $s_1s_5$ acts by $(\alpha':\epsilon':\zeta')\mapsto (-\alpha':-\epsilon':\zeta')$ and similarly $s_0s_5$ acts by $(\alpha':\epsilon':\zeta')\mapsto (\alpha':-\epsilon':-\zeta')$.  The quotient of the component under this action can be identified with
\[\{(\alpha'',\epsilon'',\zeta'') : \alpha''\epsilon''\zeta''=1\}/C_3,
\]
via the map $(\alpha':\epsilon':\zeta')\mapsto ((\alpha')^2:(\epsilon')^2:(\zeta')^2)$.

It remains to take the quotient of this by the action of $\group{u_1,u_2}$ which acts by signed permutation of the coordinates $(\alpha')^2,(\epsilon')^2,(\zeta')^2$.  The result therefore agrees with the quotient of the identity component (though geometrically one is $4$ times the area of the other), giving the quotient $\MaxTorus^w/Z_W(w)$ as the union of two equilateral triangles.

\medskip\item Case $A_2\times A_1^2$:
The fixed set is the single $2$-torus with coordinates $(\alpha, \alpha^2, 1, \epsilon^2, \epsilon, 1)$.   The centraliser  is $\group{s_0s_6}\times (\group{s_1}\wr \group{u_1})$ and the elements $s_0s_6,s_1$ both act trivially on the fixed set in the Lie algebra, and hence on its fixed torus. This leaves the action of $u_1$ to consider, which exchanges and inverts $\alpha,\epsilon$.  Replacing the variable $\epsilon$ by $\epsilon'=\epsilon^{-1}$, the quotient is therefore the symmetric product of two copies of $\TT^1$, which is a M\"obius band.

\medskip\item Case $A_2^2$:
The fixed set is a family of three $2$-tori indexed by the centre as follows:
\[
\{(\beta^{-1}, \beta, 1, \beta^{-1}, \beta, 1): \beta^3=1\}\times \{(1,1, 1, \delta, \epsilon, 1) : \delta,\epsilon\in \Ti\}
\]
Since the centre is pointwise fixed by the Weyl group, the quotient is 3 copies of the quotient of the identity component.

We write the centraliser as $(\group{s_1s_2}\wr\group{u_2})\times \group{s_4,s_5}$.  Note that $s_1s_2$ acts trivially on the fixed set, hence the action is just by the quotient group $\group{u_2}\times \group{s_4,s_5}=\group{s_4,s_5u_2}\cong \D_6$.

The element 
$s_5u_2$
takes  $(1,1, 1, \delta, \epsilon, 1)$ to $(1,1,1, \delta^{-1},\delta^{-1}\epsilon, 1)$
while the element $s_4$ acts as
    \[s_4:(1, 1, 1, \delta, \epsilon, 1)\mapsto (1, 1, 1, \delta^{-1}\epsilon, \epsilon, 1)\]

Taking a hexagonal fundamental domain for the torus, (dual to the $r_4,r_5$ lattice) we have the standard action of the dihedral group $\D_6$ yielding a $(3,4,6)$ triangle.

Hence the quotient of the fixed set by the centraliser is three copies of the triangle.

\medskip\item Case $A_3\times A_1$:
The fixed set is a single $2$-torus with coordinates $(\alpha, \alpha^2, 1, \alpha^{-2}, \epsilon, 1)$.

The centraliser is $\group{s_0s_6s_3}\times\group{s_1}\times \group{s_5}$, with $s_0s_6s_3,s_1$ acting trivially.  The element $s_5$ takes $(\alpha, \alpha^2, 1, \alpha^{-2}, \epsilon, 1)$ to $(\alpha, \alpha^2, 1, \alpha^{-2}, \epsilon^{-1}\alpha^{-2}, 1)$.  Changing coordinates to $\alpha,\epsilon'$ where $\epsilon'=\alpha\epsilon$ we see that the action takes $\epsilon'=\alpha\epsilon$ to $\alpha(\epsilon^{-1}\alpha^{-2})=(\epsilon')^{-1}$, while leaving $\alpha$ unchanged.  Hence the quotient is a cylinder.

\medskip\item Case $A_4$:
The fixed set is a single $2$-torus with coordinates $(\alpha, \beta, 1, \beta^{-1}, \beta^{-2}, 1)$.

The centraliser is $\group{s_0s_6s_3s_4}\times\group{s_1}$ with the first factor acting trivially.  The action of $s_1$ takes $\alpha$ to $\alpha^{-1}\beta$ and leaves $\beta$ unchanged.  Setting $\beta'=\alpha^{-1}\beta$
this action simply exchanges $\alpha,\beta'$ hence the quotient is the symmetric product of two copies of $\TT^1$ (a M\"obius band).

\medskip\item 
\medskip Case $D_4$ :
The fixed set is a single $2$-torus $(\alpha, \alpha^2, \alpha\epsilon, \epsilon^2, \epsilon, 1)$, which is the $\zeta=\eta=1$ component of the fixed set of $s_0s_1s_5s_3$ (case 8).  Indeed our representative is chosen so that its cube equals $s_0s_1s_5s_3$  so of course the fixed sets must be nested.  The centraliser is $\group{s_0s_1s_5T}\times\group{u_1, u_2}$ with the first factor acting trivially. As in case 8 the action of $\group{u_1, u_2}$ can be equivariantly identified with the permutation action of $\D_3$ on the $SU_3$ torus yielding the quotient as an equilateral triangle.

\medskip\item Case $D_4[a_1]$:
The fixed set is the same single $2$-torus $(\alpha, \alpha^2, \alpha\epsilon, \epsilon^2, \epsilon, 1)$ as for the previous case of type $D_4$, and here we note that the square of the representative yields the element  $s_0s_1s_5s_3$. The centraliser is the complex reflection group $G_8$ given by $\group{s_5u_3,s_5^Tu_1}$.  Again we will use the calculations from case 8, with the fixed set described as 
\[
\{(\alpha',\epsilon',\zeta') :  \alpha'\epsilon'\zeta'=1\}/C_3.
\]
As noted in that case, the elements $s_5,T$ act trivially on this set, hence we again reduce to the action of the $\group{u_3,u_1}=\group{u_1,u_2}$ giving an equilateral triangle.

\medskip\item Case $A_2^2\times A_1$:
The fixed set consists of $3$ circles indexed by a cube root of unity $\beta$ with coordinates $(\beta^2, \beta, 1, \epsilon^2, \epsilon, 1)$. Each of these circles contains one of the elements of the centre, indeed we can write elements of the fixed set in the form
\[(\beta^2, \beta, 1, \beta^2, \beta, 1)(1, 1, 1, (\epsilon')^2, \epsilon', 1).
\]
The first factor is central, hence fixed by the whole Weyl group, and the second lies in the identity component of the fixed set, and hence is fixed by $s_0,s_6,s_5,s_1,s_2$.  The centraliser is $(\group{s_0s_6}\wr\group{u_2})\times \group{s_5}$ where only the element $u_2$ acts non-trivially.  This acts on the second factor by inverting $\epsilon'$.  The quotient is therefore 3 intervals.

\medskip\item Case $A_3\times A_1^2$:
The fixed set is a family of four circles, indexed by $\zeta,\eta$ with $\zeta^2=\eta^2=1$.  We have coordinates $(\epsilon^{-1}\eta, \epsilon^{-2}\zeta, \zeta, \epsilon^2\zeta, \epsilon, \zeta)$.  We factorise this as
\[
(\epsilon^{-1}\eta, \epsilon^{-2}, \eta, \epsilon^2, \epsilon, 1)(1,\zeta,\eta\zeta,\zeta,1,\zeta).
\]
This factorisation is chosen so that for each choice of $\eta,\zeta$ we obtain a circle lying in the corresponding 2-torus appearing in case 8.

The centraliser is $\group{s_0s_6s_3s_1s_5}\times \langle s_1, T, s_5\rangle$ with the first factor acting trivially. The second factor lies in the centraliser from case 8, which contains the group $W(D_4)=\group{s_1,T,s_5,s_0}$.  The subgroup $\group{s_1,T,s_5}$ is sufficient to transitively permute the three non-identity components in case 8 and hence transitively permutes the circles in this case.

Since the centre lies in the circle $\zeta=\eta=1$ it follows that this circle is fixed by the whole of $\group{s_0s_6s_3s_1s_5}\times \langle s_1, T, s_5\rangle$.

For the non-identity components, as in case 8 consider the component $(\eta,\zeta)=(1,-1)$ which is pointwise fixed by $T$.  This component is stabilised by $\group{T,s_1s_5}\cong \D_4$ and this must be the whole stabiliser as it has index $3$ in $\group{s_1,T,s_5}\cong S_4$.  As the element $T$ acts trivially it remains to take the quotient by the action of $s_1s_5$. This fixes $(\epsilon^{-1}\eta, \epsilon^{-2}, \eta, \epsilon^2, \epsilon, 1)$ and takes $(1,\zeta,\eta\zeta,\zeta,1,\zeta)$ to
\[
(\zeta,1,1,1,\zeta,1)(1,\zeta,\eta\zeta,\zeta,1,\zeta)
\]
hence it has the effect of negating $\epsilon$.  Thus the quotient of the non-identity components yields a circle.

The quotient of the fixed set by the centraliser is therefore two circles.

\medskip\item Case $A_4\times A_1$:
The fixed set is a single circle with coordinates $(\alpha, \alpha^{-2}, 1, \alpha^{-2}, \alpha^{-4}, 1)$. As noted after Theorem \ref{the fixed sets} this circle contains the centre $\cZ$ of the Lie group, hence the action of the centraliser must be trivial.

\medskip\item Case $A_5$:
The fixed set is a family of three circles indexed by a choice $\epsilon$ of a cube root of unity and with coordinates $(\alpha, \epsilon, 1, \epsilon^2, \epsilon, 1)$. As in case 15 this can be written as a product of a circle  
with the centre so that points of the fixed set take the form
\[(\epsilon^2, \epsilon, 1, \epsilon^2, \epsilon, 1)(\alpha\epsilon, 1, 1, 1, 1, 1).
\]
The centraliser is $\group{s_0s_6s_3s_4s_5}\times \group{s_1}$ with the first factor acting trivially.  The second factor inverts $\alpha\epsilon$ hence gives a reflection on each of the three circles.

Hence the quotient is three copies of the interval.

\medskip\item[(19,20)] Cases $D_5$,$D_5[a_1]$:
In both cases the fixed set is a single circle with coordinates $(\delta^{-2},\delta^{-1},1,\delta, \delta^2,1)$. As in case 17, this must contain the three elements of the centre so is fixed pointwise, yielding a circle as the quotient.

\medskip
\addtocounter{enumi}{2}
\item \label{Fixed set A_2^3} Case $A_2^3$:
The fixed set is the finite group $C_3\times C_3\times C_3$ indexed by three choices of cube roots of $1$: $\beta, \epsilon, \zeta$. In coordinates it is given by: $(\beta^2\zeta, \beta, 1, \epsilon^2\zeta, \epsilon, \zeta)$. We can split the central points $(\epsilon^2, \epsilon, 1, \epsilon^2, \epsilon, 1)$ off from this and writing $\beta'=\beta\epsilon^2\zeta^2$ yields a factorisation as a product as follows:
\[
(\epsilon^2,\epsilon,1,\epsilon^2,\epsilon,1)(1,\zeta, 1,\zeta,1,\zeta)((\beta')^2,\beta',1,1,1,1).
\]
The first factor is the centre so fixed by everything.  We identify the fixed set as an $\F_3$ vector space $V$ with basis
\[
e_1=(\omega^2,\omega,1,\omega^2,\omega,1), e_2=(1,\omega, 1,\omega,1,\omega), e_3=(\omega^2,\omega,1,1,1,1)
\]
where $\omega=e^{2\pi i /3}$.  We take the dual basis $\{e_1^*,e_2^*,e_3^*\}$ for $V^*$.

As the first basis vector $e_1$ is fixed by the centraliser we obtain an affine action of the centraliser on the affine plane $\mathbb{A}=\{v^* \in V^*\mid \langle v^*, e_1\rangle = 1\}$.

We now consider the action of the three generators $s_0s_6, Ts_3, s_5s_4$ on $V$.  Starting with $s_0s_6$, this fixes $e_1$ and $e_3$, and takes $e_2$ to $e_1+e_2+e_3$. Similarly $s_5s_4$ fixes $e_1$ and $e_3$, and takes $e_2$ to $2e_1+e_2+e_3$.

The element $Ts_3$ fixes $e_1, e_2$, and takes $e_3$ to $e_3+2e_2$.

In matrix form the generators $s_0s_6, Ts_3, s_5s_4$ act on $V$ as:
\[
\begin{pmatrix}
1&1&0\\
0&1&0\\
0&1&1
\end{pmatrix},
\begin{pmatrix}
1&2&0\\
0&1&0\\
0&1&1
\end{pmatrix},
\begin{pmatrix}
1&0&0\\
0&1&2\\
0&0&1
\end{pmatrix}
\]

Dually, the action on $V^*$ is given by:
\[
\begin{pmatrix}
1&0&0\\
1&1&1\\
0&0&1
\end{pmatrix},
\begin{pmatrix}
1&0&0\\
2&1&1\\
0&0&1
\end{pmatrix},
\begin{pmatrix}
1&0&0\\
0&1&0\\
0&2&1
\end{pmatrix}
\]
It is not hard to see that these generate the Hessian group, i.e.\ the group of all orientation preserving affine maps on $\mathbb{A}$.

Dualising the action back to $V$ is therefore all orientation preserving linear maps which fix the vector $e_1$.  
The quotient is therefore $4$ points, corresponding to the three points on the line spanned by $e_1$ (the centre of the Lie group) and a single orbit for all the remaining vectors in $V$. 

\medskip
\item Case $A_5\times A_1$:
 In coordinates the fixed set is given by
 \[(\eta\epsilon^2, \zeta \epsilon, \zeta, \zeta \epsilon^2, \epsilon, \zeta)=(\eta\epsilon^2, \epsilon, \eta, \epsilon^2, \epsilon, 1)(1,\zeta,\zeta\eta,\zeta,1,\zeta)
 \]
 where $\epsilon$ is a cube root of unity and $\zeta^2=\eta^2=1$.  We choose this factorisation to the the copy of the Klein group $V_4$ matches with that appearing in cases 8 and 16.

The centraliser is $\group{s_0s_6s_3s_4s_5s_1}\times \langle s_1,T\rangle$ with the first factor acting trivially.  The fixed set lies in the fixed set from case 16, and the group $\langle s_1,T\rangle$ is a subgroup of the centraliser from that case.

The $\eta=\zeta=1$ part is the centre of the Lie group and is thus fixed by the centraliser.  As noted in case 16, the three other possible values for $(\eta,\zeta)$ are transitively permuted by $\group{s_1,T}$. The triple with $(\eta,\zeta)=(1,-1)$ is pointwise fixed by $T$ and its stabiliser in the group $\group{s_1,T}$ has index $3$ and is thus $\group{T}$. Hence the 6 points given by $(\eta,\zeta)=(1,1)$ and $(\eta,\zeta)=(1,-1)$ form a strict fundamental domain for the action of the centraliser.

The quotient is therefore 6 points.

\bigskip

\medskip\item[(23--25)] Cases $E_6,E_6[a_1],E_6[a_2]$:
 In each case the fixed set is the centre of the Lie group $(\epsilon^2, \epsilon, 1, \epsilon^2, \epsilon, 1)$ where $\epsilon$ ranges over the cube roots of $1$. Each of these three cases provides three points in the quotient.

\end{enumerate}

\section{Fixed sets and their quotients: the group of adjoint type}\label{section: adjoint fixed sets}

In this section we consider the case of the $E_6$ Lie group of adjoint type. The maximal torus in this case is the torus $\MaxTorus^\vee=\MaxTorus/\cZ$ where $\cZ$ denotes the centre of the Lie group of simply connected type. We will determine the quotients $(\MaxTorus^\vee)^w/Z_W(w)$ and thus establish the following:

\begin{theorem3}
\label{E6 sector theorem}
\thmiii
\end{theorem3}

As discussed in Section \ref{Component groups} there are two possibilities for the fixed sets $(\MaxTorus^\vee)^w$:  either the centre lies in the identity component $\MaxTorus^w_1$, in which case $(\MaxTorus^\vee)^w=\MaxTorus^w/\cZ$, or $\cZ$ injects into the component group $F_w$, in which case $(\MaxTorus^\vee)^w_1=\MaxTorus^w_1$.  The first case we will refer to as having ramified torus while the latter case has ramified component group.

The component group is ramified in cases $A_2^2,A_2^2\times A^1,A_5$ as well as the elliptic cases $A_2^3, A_5\times A_1, E_6,E_6[a_1],E_6[a_2]$.  In all the other cases the torus is ramified.

\subsection{Ramified torus}
In this case the fixed set $(\MaxTorus^\vee)^w$ is given by $\MaxTorus^w/\cZ$ where $\MaxTorus^w\cong \MaxTorus^w_1\times F_w$.  Hence $(\MaxTorus^\vee)^w\cong \MaxTorus^w_1/\cZ \times F_w$ (with $F_w=F_w^\vee$).

Since the action of the Weyl group arises from the conjugation action in the Lie group, this commutes with the action of the centre $\cZ$.  Hence the quotient $(\MaxTorus^\vee)^w/Z_W(w)$ is the quotient of $\MaxTorus^w$ by the action of the product $Z_W(w)\times \cZ$ and indeed we can therefore compute this as the quotient $(\MaxTorus^w/Z_W(w))/\cZ$.  The quotients $\MaxTorus^w/Z_W(w)$ were evaluated in Section \ref{fixed sets}.

Note that these quotients carry a natural metric, inherited from the Lie algebra $\t$.  The action of $\cZ$ on the quotients is induced from a translation action on $\t$ and is thus isometric.  In many cases this is sufficient to determine the quotient up to homotopy.

\begin{itemize}
    \item Case $\emptyset$: The quotient $(\MaxTorus^\vee)^w/Z_W(w)$ is $\Delta^6/\cZ$ where the contraction of $\Delta^6$ to its barycentre is equivariant, hence $\Delta^6/\cZ$ is contractible.
    
    \item Case $A_1$: The same argument shows that the quotient $\Delta^5/\cZ$ is also contractible. 
    
    \item Case $A_1^2$: The quotient $\MaxTorus^w/Z_W(w)$ in this case is a $\Delta^3$ bundle over $\TT^1$. Writing $\Delta^3$ as the join of two edges $e_0$ and $e_1$, the monodromy is given by the map inverting the edge $e_0$ and fixing $e_1$. 
    
    In logarithmic coordinates the $\Delta^3$ fibres are described by the sum of the base fibre with a point $(x_1,x_1,-\frac {x_1}{2},-\frac {x_1}{2},-\frac {x_1}{2},-\frac {x_1}{2})$.  The centre acts by translation by multiples of $(\frac23,\frac23,-\frac13,-\frac13,-\frac13,-\frac13)$ which tautologically preserves the coordinates of the fibre.
    A fundamental domain is given by $x_1\in [0,\frac13]$.  The monodromy in the identification of the fibres at $0,\frac13$ is the composition of the trivial identification of the $0,\frac43$ fibres with the (inverse of the) original monodromy in $\MaxTorus^w/Z_W(w)$.  Hence the monodromy maps agree and the quotient $(\MaxTorus^\vee)^w/Z_W(w)$ is identified with $\MaxTorus^w/Z_W(w)$.
    
    \item Case $A_2$: The quotient $\MaxTorus^w/Z_W(w)$ is $\hbox{SP}^2(\Delta^2)=(\Delta^2\times\Delta^2)/C_2$, hence the quotient $(\MaxTorus^\vee)^w/Z_W(w)$ is given by $(\Delta^2\times\Delta^2)/(C_2\times \cZ)$. The action of $C_2\times \cZ$ is isometric and hence preserves the product of the barycentres, giving a contractible quotient.
    
    \item Case $A_1^3$: The quotient is $(\Delta^2\times \Delta^1)/\cZ$. Since the action is isometric it preserves the Cartesian product structure.  Since $\cZ$ has order $3$ it preserves the $\Delta^1$ direction and the contraction to the barycentre of $\Delta^2$ times the interval $\Delta^1$ is equivariant. Hence $(\Delta^2\times \Delta^1)/\cZ$ is contractible.

    \item Case $A_2\times A_1$: The quotient $\MaxTorus^w/Z_W(w)$ is $\TT^1\times \TT^2/S_3$ giving a product $\TT^1\times \Delta^2$. The action of $\cZ$ acts diagonally by rotations on both factors. This action is fixed point free and yields a solid torus which carries a Seifert fibration over the cone $\Delta^2/\cZ$, \cite{scott1983geometries}.
    
    \item Case $A_3$: The quotient $\MaxTorus^w/Z_W(w)$ in this case is a triangle bundle over $\TT^1$, with the triangle obtained as the quotient of a square by a diagonal reflection.  The monodromy is obtained from the rotation of the square by $\pi$, which yields a diagonal flip on the triangle. The action of $\cZ$ is to translate by $1/3$ along the circle while again giving a rotation by $\pi$ on the square. Note that the cube of this element agrees with the monodromy of the original bundle as required. Hence the quotient $(\MaxTorus^\vee)^w/Z_W(w)$ is identified with $\MaxTorus^w/Z_W(w)$.
    
    \item Case $A_1^4$: The action of $\cZ$ on $\Delta^2\sqcup \Delta^2$ must preserve the components fixing the barycentre of each, and the contraction to the barycentres is equivariant.

    \item Case $A_2\times A_1^2$: The action of $\cZ$ on the symmetric product $\hbox{SP}^2(\TT^1)$ is given by the diagonal action of $\cZ$ on the two copies of the circle $\TT^1$.  Since this is a fixed point free action on the M\"obius band, the quotient is again a M\"obius band.
    
    \item Case $A_3\times A_1$: The quotient $\MaxTorus^w/Z_W(w)$ is $\TT^2/C_2=\Delta^1\times \TT^1$. The action of $\cZ$ rotates the $\TT^1$ factor, again yielding a copy of $\Delta^1\times \TT^1$.
    
    \item Case $A_4$: As in case $A_2\times A_1^2$ the centre $\cZ$ acts diagonally on $\hbox{SP}^2(\TT^1)$ giving quotient a M\"obius band.

    \item Cases $D_4$ and $D_4[a_1]$: In these cases the quotient $\MaxTorus^w/Z_W(w)$ are equal, giving a copy of $\Delta^2$. The group $\cZ$ acts by rotation on this giving a cone $\Delta^2/\cZ$.
    
    \item Case $A_3\times A_1^2$: The action of $\cZ$ must preserve the two components $\TT^1\sqcup\TT^1$, acting isometrically and fixed point freely on each circle, hence the quotient is a disjoint union $\TT^1\sqcup\TT^1$.

    \item Cases $A_4\times A_1, D_5, D_5[a_1]$: In each of these cases the centre acts fixed-point freely on $\TT_1$ and the quotient is again a circle.
    
    \end{itemize}

\subsection{Ramified component groups}

We denote elements of the centre $\cZ$ by $z_\omega=(\omega^2,\omega,1,\omega^2,\omega,1)$ where $\omega$ is a cube root of $1$.

\begin{itemize}
    \item Case $A_2^2$:
    
    Recall that $w=s_0s_6s_1s_2$ and $\MaxTorus^w=\MaxTorus^w_1\times \cZ$ where (in standard coordinates on $\TT^6$) the identity component is
    \[
    \MaxTorus^w_1=\{(1,1,1,\delta,\epsilon,1):\delta,\epsilon \in \Ti\}.
    \]
    For $\omega$ a cube root of $1$, the action of $w$ takes $p_\omega=(1,\omega^2,1,1,1,\omega)$ to 
    \[
    (\omega^2,1,1,\omega^2,\omega,\omega)=z_\omega p_\omega
    \]
    so $p_\omega$ is fixed by $w$ up to the action of $\cZ$.
    
    When $\omega\neq 1$ the point $p_\omega$ does not lie in $\MaxTorus^w$, hence the image of $p_\omega$ in the quotient by $\cZ$ lies in a non-identity component of $(\MaxTorus^\vee)^w$. We can therefore index the components of $(\MaxTorus^\vee)^w$ by $\omega$.
    
    We then have
    \[
    (\MaxTorus^\vee)^w=(\{p_\omega : \omega^3=1\}\times \MaxTorus^w)/\cZ,
    \]
    and taking a fundamental domain for the action of $\cZ$, we can identify $(\MaxTorus^\vee)^w$ with
    \[
    \{p_\omega : \omega^3=1\}\times \{(1,1,1,\delta,\epsilon,1):\delta,\epsilon \in \Ti\}.
    \]
    The centraliser of the element $w$ is $\group{s_1s_2}\wr\group{u_2}\times \group{s_4,s_5}$ and we now consider its action on $(\MaxTorus^\vee)^w$.
    
    We first consider the action of the normal subgroup $\group{s_1s_2,s_0s_6}$ in the centraliser.  (Recall that $(s_1s_2)^{u_2}=s_0s_6$.) Since $s_1s_2s_0s_6$ acts trivially it suffices to consider the action of $s_1s_2$.  This does not preserve the fundamental domain for the action of $\cZ$, so we act by $s_1s_2$ and then by the central element $z_\omega$. This takes $p_\omega$ to $(1,\omega^2,1,\omega^2,\omega,\omega)=p_\omega(1,1,1,\omega^2,\omega,1)$.
    
    Hence the action of $s_1s_2$ preserves components of $(T^\vee)^w$ and acts on the component indexed by $\omega$ by multiplication by $(1,1,1,\omega^2,\omega,1)\in (\MaxTorus^\vee)_1^w=\MaxTorus_1^w$.
    
    Recall that the action of the $\D_6$ group $\group{u_2,s_4,s_5}=\group{s_4,u_2s_5}$ on $\MaxTorus^w_1$ is given by the standard dihedral action by the group of symmetries of the hexagonal fundamental domain for the torus, which we will denote $\hexagon$. The element $s_5$ acts trivially while $u_2$ acts to invert the $\epsilon$ coordinate. For the identity component $ (\MaxTorus^\vee)_1$, indexed by  $\omega=1$, the element $s_1s_2$ acts trivially, and the quotient by the centraliser is a $(2,4,6)$ triangle as in the case of $\MaxTorus^w$.
    
    When $\omega$ is a nontrivial cube root of the element $1$,  our $\D_6$ group still preserves the component, but $s_1s_2$ no longer acts trivially so we first factor out its action, which we recall is given by multiplication by  $(1,1,1,\omega^2,\omega,1)$.  This yields the torus with dual hexagon as fundamental domain and we denote this fundamental domain by $\varhexagon$. So the quotient of this component by the centraliser is given by taking the dual action of $\D_6$ on the dual hexagon $\varhexagon$. But these actions are equivalent so we again obtain as quotient a $(2,4,6)$ triangle, here with $1/3$ the area of the $\omega=1$ case, and the quotient $(\MaxTorus^\vee)^w$ by the centraliser of $w$ is three $(2,4,6)$ triangles.

    \item Case $A_2^2\times A_1$:  Here the representative is $w=s_0s_6s_5s_1s_2$.  Since $s_5$ acts trivially on $p_\omega$, we have $w\cdot p_\omega=z_\omega p_\omega$ using the previous case. We thus have
    \[
    (\MaxTorus^\vee)^w=(\{p_\omega : \omega^3=1\}\times \MaxTorus^w)/\cZ.
    \]
    where $\MaxTorus^w=\MaxTorus^w_1\times \cZ$ and $\MaxTorus^w_1=\{(1,1,1,\epsilon^2,\epsilon,1) : \epsilon \in \Ti\}$.
    
    Taking a fundamental domain for the action of $\cZ$, we can identify $(\MaxTorus^\vee)^w$ with
    \[
    \{p_\omega : \omega^3=1\}\times \{(1,1,1,\epsilon^2,\epsilon,1):\epsilon \in \Ti\}.
    \]
    The centraliser is $\group{s_1s_2}\wr\group{u_2}\times \group{s_5}$.
    
    As noted in the $A_2^2$ case, for $\omega$ a cube root of $1$, the action of $s_1s_2$ followed by $z_\omega$ takes $p_\omega$ to $(1,\omega^2,1,\omega^2,\omega,\omega)$, so in $(\MaxTorus^\vee)^w$ the element $s_1s_2$ has the effect of preserving the three components and on each it multiplies the coordinate $\epsilon$ by $\omega$.
    
    It is easy to see that the generator $s_5$ acts trivially on the above fundamental domain.  The generator $u_2$ also fixes $p_\omega$ because it interchanges $r_2$ and $-r_6$. As in the case of $\MaxTorus^w$ the element $u_2$ acts on the second factor by inverting $\epsilon$.
    
    Hence the quotient of each component is an interval: the element $s_1s_2$ can act either trivially or by rotation on the circle. For $\omega=1$ we have the former, just leaving the action of $u_2$. The latter occurs when $\omega\neq 1$, and we have the dihedral action of the group $(\group{s_1s_2}\wr\group{u_2})/\group{s_1s_2(s_1s_2)^{u_2}}\cong \D_3$.

    \item Case $A_5$: Consider the action of the representative $w=s_0s_6s_3s_4s_5$ on the point $q_\omega=(1, 1, 1, 1, \omega^2, \omega^2)\in \MaxTorus$, where $\omega$ is a cube root of unity. The action of $w$ translates this by the element $z_w$ in the centre, so it represents a fixed point in the dual torus $\MaxTorus^\vee$. Hence the fixed set under the dual action of $w$ is given by
    \[
    (\MaxTorus^\vee)^w=(\{q_\omega : \omega^3=1\}\times \MaxTorus^w)/\cZ.
    \]
where $\MaxTorus^w=\MaxTorus^w_1\times \cZ$ and $\MaxTorus^w_1=\{(\alpha, 1, 1, 1, 1, 1)\mid\alpha\in \TT^1\}$.
As in the previous case, by taking a fundamental domain for the action of $\cZ$,  we can identify the fixed set with the set 
  \[
    \{q_\omega : \omega^3=1\}\times \{(\alpha, 1, 1, 1, 1, 1)\mid\alpha\in \TT^1\}.
    \]

    The centraliser is $\group{s_0s_6s_3s_4s_5}\times \group{s_1}$ with the first factor acting trivially.  The element $s_1$ preserves the three circles and in each case acts by a reflection taking $\alpha$ to $\alpha^{-1}$. The quotient is therefore, again, three copies of the interval.

    \item Case $A_2^3$: The dual case is considered in detail in the derivation of the action of the centraliser on $\MaxTorus^w$, so we will not repeat it here. In summary, the dual fixed set can be identified with a 3-dimensional vector space over $\F_3$, equipped with a quotient to $\F_3$.  The cosets of the kernel give two affine planes and a linear plane, and the action of the centraliser is given by the action of the Hessian group by orientation preserving affine maps on the affine planes, with its linear part on the linear plane. The quotient is therefore a set of four points.
    
    For the sake of completeness we note that $w=s_0s_6s_1s_2s_5s_4$ takes $p_\omega$ to $(\omega^2,1,1,\omega^2,\omega,\omega)=z_\omega p_\omega$.  The fixed set $\MaxTorus^w$ is
    \[
    \{(1,\zeta, 1,\zeta,1,\zeta)(\beta^2,\beta,1,1,1,1) : \zeta^3=\beta^3=1\}\times \cZ.
    \]
    The fixed set $(\MaxTorus^\vee)^w_1$ is therefore identified with a fundamental domain
    \[
    \{(1,\omega^2,1,1,1,\omega)(1,\zeta, 1,\zeta,1,\zeta)(\beta^2,\beta,1,1,1,1) : \omega^3=\zeta^3=\beta^3=1\}.
    \]
    
    \item Case $A_5\times A_1$: Consider the action of the representative $w=s_0s_6s_3s_4s_5s_1$ on the point $(1, 1, 1, 1, \omega^2, \omega^2)\in \MaxTorus$, where $\omega$ is a cube root of unity.  As noted in the $A_5$ case, this element is fixed by $s_1$ and is fixed up to the action of $\cZ$ by $s_0s_6s_3s_4s_5$.  Hence 
    \[
    (\MaxTorus^\vee)^w=(\{q_\omega : \omega^3=1\}\times \MaxTorus^w)/\cZ.
    \]
    Here elements of $\MaxTorus^w$ have the form $(\eta \epsilon^2,\zeta\epsilon,\zeta,\zeta\epsilon^2,\epsilon,\zeta)$ where $\epsilon^3=\zeta^2=\eta^2=1$.  A fundamental domain for the action of $\cZ$ is given by
    \[
    \{q_\omega : \omega^3=1\}\times \{(\eta,\zeta,\zeta,\zeta,1,\zeta): \zeta^2=\eta^2=1\}.
    \]
    The centraliser is $\group{w,s_1,T}$ where $w$ acts trivially (on the quotient by $\cZ$). Clearly the element $s_1$ fixes the first factor, and $T$ also fixes this because $T$ fixes the vector $r_5+r_6$.  We note that the  second factor is obtained by applying the exponential map to $\frac 12 \Z r_1+ \frac 12\Z r_T$ and $r_1,-r_T$ have inner product $-1$ (giving the group $\group{s_1,T}$ of type $A_2$).
    
    The action of $\group{s_1,T}$ on pairs $(\eta,\zeta)$ is thus the standard action $s_1:(\eta,\zeta)\to (\eta^{-1}\zeta,\zeta)=(\eta\zeta,\zeta)$ and $T:(\eta,\zeta)\to (\zeta,\zeta)=(\eta,\eta\zeta)$.
    
    Hence the quotient is $6$ points: a fundamental domain is given by $\zeta=1$.
    
    \item Cases $E_6,E_6[a_1],E_6[a_2]$. In each of these cases the fixed set in $\MaxTorus$ consists entirely of the centre of the Lie group, and is fixed pointwise. It follows from the discussion in Section \ref{Component groups}, Remark \ref{ramified component remark}, that the fixed set in $\MaxTorus^\vee$ can be identified with the Pontryagin dual of the centre, which is again cyclic of order $3$, and that the centraliser acts on this by the (dual of) the trivial action. All that remains, for the sake of completeness is to identify the three fixed point groups in these cases.
    
    In the $E_6$ case the element $w=s_1s_2s_3s_4s_5s_6$ takes $(\omega,1,1,\omega,1,1)$ to $(1,\omega,1,1,\omega,1)=z_\omega(\omega,1,1,\omega,1,1)$ for $\omega$ a cube root of unity, so the fixed set is identified with $\{(1,\omega,1,1,\omega,1):\omega^3=1\}$.
    
    For $E_6[a_1]$, the element $w=s_1s_2s_3s_4s_5s_6^{s_3}$ takes $(1,\omega,1,\omega,1,\omega^2)$ to $z_\omega(1,\omega,1,\omega,1,\omega^2)$, giving fixed set $\{(1,\omega,1,\omega,1,\omega^2):\omega^3=1\}$.
    
    Finally for $E_6[a_2]$, the element $w=s_6s_2s_0^Ts_1^Ts_4s_3$ takes $(1,\omega,1,1,1,\omega^2)$ to $z_\omega(1,\omega,1,1,1,\omega^2)$, giving fixed set $\{(1,\omega,1,1,1,\omega^2):\omega^3=1\}$.
\end{itemize}

\begin{table}[h]
\caption{Fixed sets in $\MaxTorus^\vee$ for ramified component group cases}\label{ramified component groups}
\begin{center}
\rowcolors{1}{}{lightgray}
\begin{tabular}{|l|l|}
\hline
\rowcolor{white}\multicolumn{1}{|c|}{\hbox{Conjugacy class}}&\multicolumn{1}{|c|}{\hbox{Fixed sets in $\MaxTorus^\vee$} lifted to $\MaxTorus$}\\
\hline
$A_2^2$&$ \{(1,1,1,\delta,\epsilon,1):\delta,\epsilon \in \Ti\}\times \{(1, \omega^2,1,1,1,\omega) : \omega^3=1\}$\\
$A_2^2\times A_1$&$ \{(1,1,1,\epsilon^2,\epsilon,1):\epsilon \in \Ti\}\times \{(1, \omega^2,1,1,1,\omega) : \omega^3=1\}$\\
$A_5$&$ \{(\alpha, 1, 1, 1, 1, 1)\mid\alpha\in \TT^1\}\times \{(1, 1, 1, 1, \omega^2, \omega^2) : \omega^3=1\}$\\
$A_2^3$ & $\{(\beta^2,\beta\zeta, 1,\zeta,1,\zeta) : \zeta^3=\beta^3=1\}\times\{(1,\omega^2,1,1,1,\omega)\mid \omega^3=1\}$ \\
$A_5\times A_1$&$ \{(\eta,\zeta,\zeta,\zeta,1,\zeta): \zeta^2=\eta^2=1\}\times \{(1, 1, 1, 1, \omega^2, \omega^2) : \omega^3=1\} $\\
$E_6$&$\{(1,\omega,1,1,\omega,1):\omega^3=1\}$\\
$E_6[a_1]$&$\{(1,\omega,1,\omega,1,\omega^2):\omega^3=1\}$\\
$E_6[a_2]$&$\{(1,\omega,1,1,1,\omega^2):\omega^3=1\}$ \\
\hline\end{tabular}
\end{center}
\end{table}
\newpage

\section{Applications}\label{applications}   There are two major applications for the results in this article.   

\subsection{Extended affine Weyl groups of type $E_6$}   Our first application is the computation of the rational $K$-theory of the group $C^*$-algebras for extended affine Weyl groups. From \cite{niblo2018poincare} we know that Langlands duality induces an isomorphism  in $K$-theory, and the computations given here make it possible to exhibit explicit classes and to visualise this isomorphism.

\def\thmi{Let $W_a'$ be an extended affine Weyl group of one of the two Lie groups of type $E_6$ (simply connected or adjoint type).   Up to torsion, $K_\ast(C_r^\ast\, W_a')$ is 
\begin{eqnarray*}
& &\Z^{47} \quad \quad \textrm{in dimension } 0\\
& & \Z^{11} \quad \quad \textrm{in dimension }  1.
\end{eqnarray*}}

\begin{theorem}\label{K-theory calculation}
\label{thm1}\thmi
\end{theorem}

\begin{proof}
 Let $G$ be a compact connected simply-connected semisimple  Lie group of type $E_6$, with centre $\cZ$. Let $\MaxTorus$ be a maximal torus in $G$, and let $\Gamma$ be the kernel of the exponential map $\exp : \mathfrak{t} \to T$, which is simply the root lattice since $G$ is simply connected.

The Langlands dual $G^\vee=G/\cZ$ is the adjoint form and its extended affine Weyl group is  
\[
W_a'(G^\vee): = \Gamma^\vee \rtimes W
\]
where $\Gamma^\vee$ is the dual lattice. Its group $C^*$-algebra $C^*(W_a'(G^\vee))$ is isomorphic to $C(\MaxTorus)\rtimes W$ and by the Green-Julg theorem its $K$-theory is isomorphic to $K^*_W(\MaxTorus)$

Now we apply the equivariant Chern character \cite {baum1988chern} for the discrete group $W$:
\[
K_W^*(\MaxTorus ) \otimes_{\Z} \C \simeq H^*(\MaxTorus \q W; \C)
\]
where the  geometric extended quotient $\MaxTorus \q W$ is the quotient of the inertia space
\[
\widetilde{\MaxTorus } = \{(w,x)\in W\times \MaxTorus : wx=x\}
\]
by the action $g(w,x)=(w^g,gx)$. The quotient is given by the formula
\[
\MaxTorus \q W : = \bigsqcup \MaxTorus ^w/Z_W(w) 
\]
with one $w$ chosen in each $W$-conjugacy class.   Putting all this together, we obtain 
\[
K_* \,C^*(W'_a(G^\vee)) \otimes_{\Z} \C \simeq \bigoplus H^*(\MaxTorus^w/Z_W(w); \C)
\]
The 25 sectors  
are listed in the right-hand column of Table \ref{Fixed sets and centralisers}, and the $K$-theory is readily computed.   

For the extended affine Weyl group $W_a'(G)$ the $K$-theory is obtained by replacing the sectors $\MaxTorus^w/Z_W(w)$ by $(\MaxTorus^\vee)^w/Z_W(w)$ each of which is homotopic to the corresponding sector for the action of $W$ on $\MaxTorus$ by Theorem \ref{E6 sector theorem}.
\end{proof}

\subsection{$p$-adic groups of type $E_6$}

The Baum-Connes conjecture does not help with computing the $K$-theory of $p$-adic groups since the left hand side is not itself tractable. The ABPS-framework developed in \cite{aubert2015conjectures} is designed to 
fix this, providing a more effective approach to $p$-adic groups.   Indeed, Conjecture 5, \emph{ibid}  provides a much finer and more precise formula in $K$-theory than Baum-Connes alone provides.

Let $\cG$ denote a split group of type $E_6$ over a $p$-adic field; the group may be of adjoint type or simply connected.   
Let $C^\ast_r(\cG)$ denote the reduced $C^\ast$-algebra of $\cG$.   The reduced $C^*$-algebra admits the Bernstein decomposition
\[
C^*_r(\cG) = \bigoplus_{\fs \in \fB(\cG)} C^*_r(\cG)^\fs
\]
where $\fB(\cG)$ is the Bernstein spectrum of $\cG$.   Each point $\fs$ in the Bernstein spectrum is an equivalence class of cuspidal pairs
$(M, \sigma)$ where $M$ is a Levi subgroup of $\cG$ and $\sigma$ is an irreducible cuspidal representation of $\cM$.  In particular, we have the 
\emph{Iwahori point}  $\mathfrak{i}$ defined by the pair $(\cT_p, 1)$ where $\cT_{p}$ is a maximal torus of (the $p$-adic group) $\cG$ and $1$ is the trivial representation of $\cT_p$.   
We will write
\[
\fA = C^*_r(\cG)^{\mathfrak{i}}.
\]
This is called  the reduced Iwahori-spherical $C^*$-algebra.  The spectrum of the $C^*$-algebra $\fA$ comprises all irreducible 
tempered representations of $\cG$ which admit a nonzero Iwahori-fixed vector. 

Let $\cG^\vee$ denote the (complex) Langlands dual of $\cG$, let $\cT_\C$ denote a maximal torus in $\cG^\vee$, and let $\cT$ denote the maximal compact subgroup of 
$\cT_\C$.     

According to \cite[Eqn.(4.9)]{aubert2015conjectures}, we have
\[
K_j(\fA) \otimes_{\Z} \Q \simeq K^j_W(\cT) \otimes_{\Z} \Q
\]

The results in $\S 5.1$ now lead immediately to the following answer for the $K$-theory of $\fA$:
\[
K_0(\fA)\otimes_{\Z} \C  = \C^{47}, \quad K_1(\fA) \otimes_{\Z} \C = \C^{11}.
\]

From the point of view of noncommutative geometry, the $C^*$-algebra $\mathfrak{A}$ behaves, 
at the level of $K$-theory (after tensoring by $\C$), as if its spectrum was equal to the extended quotient $\MaxTorus\q W$.

\newpage

 \newgeometry{margin=0.3in}
 \begin{landscape}
 \section{Power relations between conjugacy classes and reflection structures of centralisers}
 
 $G_5, G_6, G_8, G_{25}, G_{28}, G(p,q,r)$ are complex reflection groups as classified by Shephard-Todd.
 \begin{figure}[ht]
   \centering
   \includegraphics[width=11in]{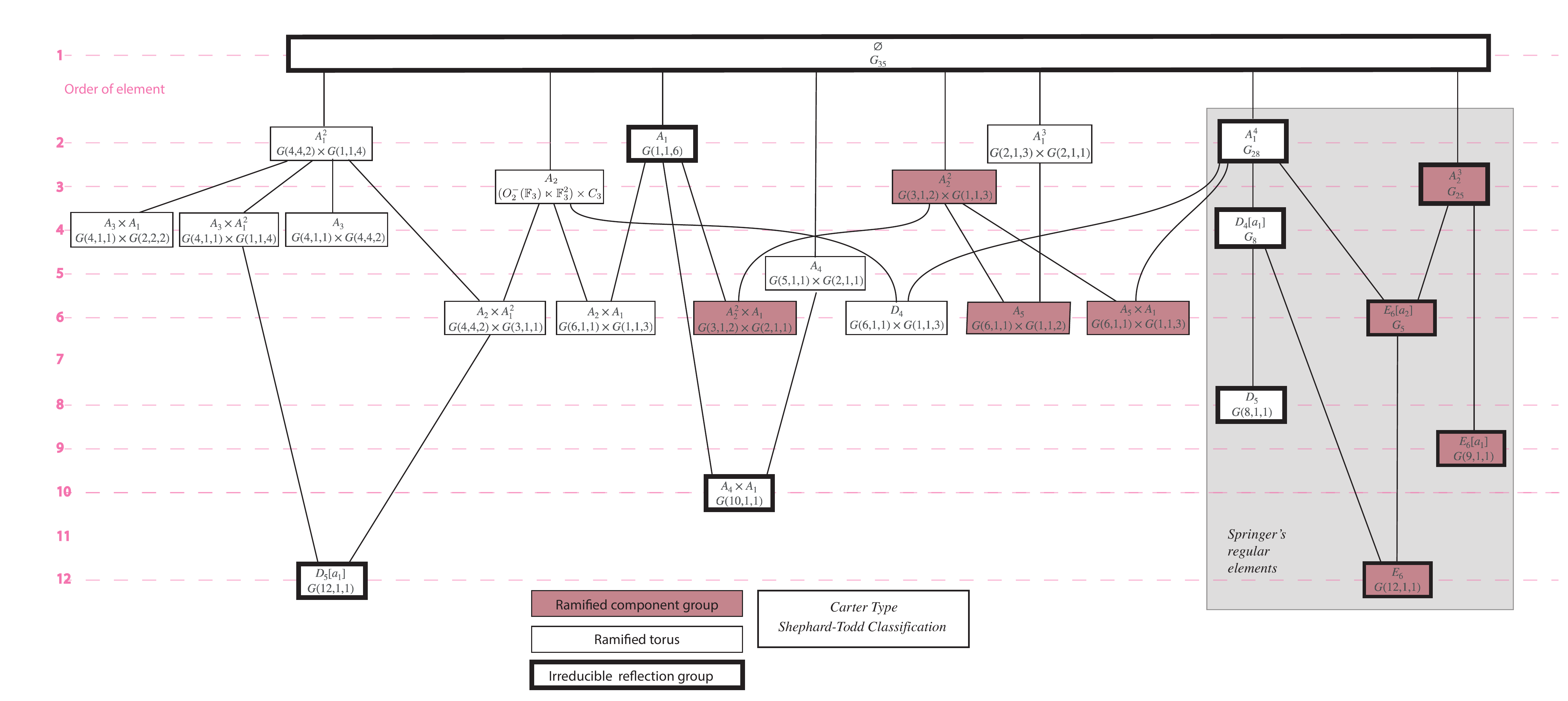} 
   \caption{Inclusions among the centralisers induced by power relations of conjugacy class representatives.  Layers show the order of the representatives, indicating the relevant power relation.}
   \label{fig:centraliserinclusions}
 \end{figure}
 \end{landscape}
 \restoregeometry

\skiptoc

\end{document}